\documentclass[11pt]{article}
\usepackage[english]{babel}
\usepackage[T1]{fontenc}
\usepackage[utf8]{inputenc}
\usepackage{amssymb}
\usepackage{amsfonts}
\usepackage{amsmath}
\usepackage{mathrsfs}
\usepackage{graphicx}
\usepackage{amsbsy}
\usepackage{subcaption}
\usepackage{theorem}
\usepackage{adjustbox}
\usepackage{color}
\usepackage[colorlinks=true]{hyperref}
\usepackage{epstopdf}
\usepackage{dsfont}
\usepackage[titletoc,toc,title]{appendix}
\usepackage[normalem]{ulem}
\usepackage{cellspace,soul}%
\newtheorem{theorem}{Theorem}[section]
\newtheorem{lemma}{Lemma}

\newtheorem{proposition}[theorem]{Proposition}
\newtheorem{remark}[theorem]{Remark}

\def\beq{\begin{equation}\displaystyle}
\def\eeq{\end{equation}}
\def\bel{\begin{equation} \displaystyle \begin{array}{l} }
\def\eel{\end{array} \end{equation} }
\def\bell{\begin{equation} \displaystyle \begin{array}{ll}  }
\def\eell{\end{array} \end{equation} }

\def\bea{\begin{eqnarray}}
\def\eea{\end{eqnarray} }
\def\bean{\begin{eqnarray*}}
\def\eean{\end{eqnarray*} }
\newenvironment{proof}{\noindent{\bf Proof.~}}
{{\mbox{}\hfill {\small \fbox{}}\\}}
\def\RR{\mathbb{R}}

\newcommand{\bepa}{\left\{ \begin{array}{l}}
\newcommand{\eepa}{\end{array} \right.}
\newcommand{\p}{\partial}
\newcommand{\f}{\frac}

\def\eps{\varepsilon}

\newcommand{\df}{\displaystyle\frac}

\newcommand{\0}{\mathbf{0}}

\newcommand{\EE}{\mathbf{E}}
\newcommand{\Ep}{\mathbf{E}_+}
\newcommand{\Em}{\mathbf{E}_-}
\newcommand{\lowEE}{\underline{\mathbf{E}}}

\newcommand{\ovest}[1]{\widehat{\overline{#1}}}
\newcommand{\undest}[1]{\widehat{\underline{#1}}}

\newcommand{\calN}{\mathcal{N}}

\newcommand{\tilmu}{\widetilde{\mu}}
\newcommand{\lowgam}{\underline{\gamma}}

\newcommand{\tcm}{\textcolor{magenta}}

\newcommand{\sgn}{\textrm{sgn}}

\begin{document}

\title{On the use of the sterile insect technique or the incompatible insect technique to reduce or eliminate mosquito populations}
\author{Martin Strugarek$^{1,2}$, Herv\'{e} Bossin$^{3}$, and Yves Dumont$^{4,5, 6}$\\\\
$^1$ \small AgroParisTech, 16 rue Claude Bernard, 75231 Paris Cedex 05 - France
\\
$^2$ \small Sorbonne Université, Université Paris-Diderot SPC, CNRS, INRIA,\\ \small Laboratoire Jacques-Louis Lions, équipe Mamba, F-75005 Paris\\
$^3$ \small Institut Louis Malardé, Unit of Emerging Infectious Diseases, Papeete 98713, Tahiti, French Polynesia\\
$^4$ \small CIRAD, Umr AMAP, Pretoria, South Africa\\
$^5$ \small AMAP, Univ Montpellier, CIRAD, CNRS, INRA, IRD, Montpellier, France\\
$^6$\small University of Pretoria, Department of Mathematics and Applied Mathematics,\\ \small Pretoria, South Africa}
\maketitle
\begin{abstract}
 Vector control is critical to limit the circulation of vector-borne diseases like chikungunya, dengue or zika which have become important issues around the world. Among them the Sterile Insect Technique (SIT) and the Incompatible Insect Technique (IIT) recently aroused a renewed interest. In this paper we derive and study a minimalistic mathematical model designed for {\itshape Aedes} mosquito population elimination by SIT/IIT. Contrary to most of the previous models, it is bistable in general, allowing simultaneously for elimination of the population and for its survival. We consider different types of releases (constant, periodic or impulsive) and show necessary conditions to reach elimination in each case. We also estimate both sufficient and minimal treatment times. Biological parameters are estimated from a case study of an {\itshape Aedes polynesiensis} population, for which extensive numerical investigations illustrate the analytical results. The applications of this work are two-fold: to help identifying some key parameters that may need further field investigations, and to help designing release protocols. 
\end{abstract}


\section*{Introduction}
Sterile insect technique (SIT) is a promising technique that has been first studied by E. Knipling and collaborators and first experimented successfully in the early 50's by nearly eradicating screw-worm population in Florida. Since then, SIT has been applied on different pest and disease vectors, like fruit flies or mosquitoes (see \cite{SIT} for an overall presentation of SIT and its applications). The classical SIT relies on the mass releases of males sterilized by ionizing radiations. The released sterile males transfer their sterile sperms to wild females, which results in a progressive reduction of the target population. For mosquito control in particular, new approaches stemming from SIT have emerged, namely the RIDL technique, and the {\itshape Wolbachia} technique. {\itshape Wolbachia} is a bacterium that infects many Arthropods, and among them some mosquito species in nature. It was discovered in 1924 \cite{W}. Since then, particular properties of these bacteria have been unveiled. One of of these properties is particularly useful for vector control: the cytoplasmic incompatibility (CI) property \cite{Sinkin2004,Bourtzis2008}. CI can serve two different control strategies:
\begin{itemize}
\item Incompatible Insect Technique (IIT): the sperm of W-males (males infected with CI-inducing {\itshape Wolbachia}) is altered so that it can no longer successfully fertilize uninfected eggs. This can result in a progressive reduction of the target population. Thus, when only W-males are released the IIT can be seen as classical SIT. This also supposes that releases are made regularly until extinction is achieved (when possible) or until a certain threshold is reached (in order to reduce exposure to mosquito bites and the epidemiological risk).
\item Population replacement: when males and W-females are released in a susceptible (uninfected) population, due to CI, W-females will typically produce more offspring than uninfected females. Because {\itshape Wolbachia} is maternally inherited this will result in a population replacement by {\itshape Wolbachia} infected mosquitoes (such replacements or invasions have been observed in natural population, see \cite{rasgon2003wolbachia} for the example of Californian {\itshape Culex pipiens}). It has been showed that this technique may be very interesting with \textit{Aedes aegypti}, shortening their lifespan (see for instance \cite{Schraiber2012}), or more importantly, cutting down their competence for dengue virus transmission \cite{Moreira2009}. However, it is also acknowledged that {\itshape Wolbachia} infection can have fitness costs, so that the introgression of {\itshape Wolbachia} into the field can fail \cite{Schraiber2012}.
\end{itemize}
Based on these biological properties, classical SIT and IIT (see \cite{dufourd2011,dumont2012,dufourd2013,Li2015,Li2017} and references therein) or population replacement (see \cite{Farkas2010,Fen.Solving,Schraiber2012,HugBri.Modeling,Farkas2017,nadin2017,Strugarek2018} and referencs therein) have been modeled and studied theoretically in a large number of papers in order to derive results to explain the success or not of these strategies using discrete, continuous or hybrid modeling approaches, temporal and spatio-temporal models. Recently, the theory of monotone dynamical systems has been applied efficiently to study SIT \cite{anguelov2012}
 or population replacement \cite{Sallet2015,Bliman2017} systems.
 
 Here, we derive a monotone dynamical system to model the release and elimination process for SIT/IIT. The analytical study of this model is complemented by a detailed parametrization to describe real-life settings, and a thorough investigation of numerical scenarios.

The outline of the paper is as follows. First, we explain in Section \ref{sec:modeling} the biological situation we consider and the practical questions we want to answer, namely: how to quantify the release effort required to eliminate an {\itshape Aedes} population using SIT/IIT, with particular emphasis on the timing and size of the releases.
We also justify our modeling choices and give value intervals deduced from experimental results for most biological parameters in Table \ref{table:bioparam}.
Then, we perform the theoretical analysis of a simple, compartimentalized population model featuring an Allee effect and a constant sterilizing male population in Section \ref{sec:theory}. Proposition \ref{prop:separatrice} gives the bistable asymptotic behavior of the system, and introduces the crucial separatrix between extinction and survival of the population. We also provide analytical inequalities on the entrance time of a trajectory into the extinction set (Proposition \ref{prop:upperboundtauMi}), which is extremely useful to understand what parameters are really relevant and how they interact. 
We then analyze the model as a control system, after adding a release term.
Finally, Section~\ref{sec:num} exposes numerical investigations of the various models, and applies them to a specific case study (a pilot field trial led by one of the authors).

In general, all mathematical results are immediately interpreted biologically. To keep the exposition as readable as possible, we gather all technical developments of the proofs into Appendices.

\section{Modeling and biological parameter estimation}
\label{sec:modeling}
\subsection{Modeling context}

Our modeling effort is oriented towards an understanding of large-scale time dynamics of a mosquito population in the {\itshape Aedes} genus exposed to artificial releases of {\itshape sterilizing males}. These males can be either sterilized by irradiation (Sterile Insect Technique approach) or simply have a sterile crossing with wild females due, for instance, to incompatible strains of {\itshape Wolbachia} bacteria (Incompatible Insect Technique approach). In either techniques (SIT or IIT), the released males are effectively {\itshape sterilizing} the wild females they mate with.

Eggs from mosquitoes of various species in the {\itshape Aedes}  genus  resist to dessication and can wait for months before hatching. Due to rainfall-dependency of natural breeding sites availability, this feature allows for maintaining a large quiescent egg stock through the dry season, which triggers a boom in mosquito abundance when the rainy season resumes.
For the populations we model here, natural breeding sites are considered to be prominent, and therefore it is absolutely necessary that our models take the egg stock into account.

We use a system biology approach to model population dynamics. In the present work we neglect the seasonal variations and assume all biological parameters to be constant over time. 

Our first compartmental model features egg, larva, adult male and adult female (fertile or sterile) populations. Most transitions between compartments are assumed to be linear. Only three non-linear effects are accounted for. 

First, the population size is bounded due to an environmental carrying capacity for eggs, which we model by a logistic term.
Secondly, the sterilizing effect creates two sub-populations among inseminated females. Some are inseminated by wild males and become fertile while the others are inseminated by sterilizing males and become sterile. Hence the relative abundance (or more precisely the relative mating power) of sterilizing males with respect to wild males must appear in the model, and is naturally a nonlinear ratio. Many other parameters may interfere with the mating process for {\it Aedes} mosquitoes, but this process is not currently totally understood in particular from the male point of view \cite{Lees2014,Oli.Male}, and we stick here to the simplest possible modeling.
Thirdly, as a result of sterilizing matings, we expect that the male population can drop down to a very low level. We introduce an Allee effect which come into play in this near-elimination regime. This effect reduces the insemination rate at low male density, as a consequence of difficult mate-finding. It can also be interpreted as a quantification of the size of the mating area relative to the total size of the domain, and compensates in some ways the intrinsic limitations of a mean-field model for a small and dispersed population (cf. \cite{Dur.Importance} and see Remark \ref{rem:Allee}).
Indeed, we model here temporal dynamics by neglecting spatial variations and assuming homogeneous spatial distribution of the populations.
In nature, the distribution of {\itshape Aedes} mosquitoes is mostly heterogenous, depending on environmental factors such as vegetation coverage, availability of breeding containers and blood hosts. The proposed simplified homogenous model will thus be exposed to potential criticism.

\subsection{Models and their basic properties}

We denote by $E$ the eggs, $L$ the larvae, $M$ the fertile males, $F$ the fertile females and $F_{st}$ the sterile females (either inseminated by sterilizing males or not inseminated at all, due to male scarcity). The time-varying sterilizing male population is denoted $M_i$. We use Greek letters $\mu$ for mortality rates, $\nu$ for transition rates and denote fecundity by $b$ (viable eggs laid per female and per unit of time) and egg carrying capacity by $K$. The full model reads:
\beq
\bepa
\df{d E}{dt} = b F ( 1 - \f{E}{K}) - (\widetilde{\nu}_E + \mu_E) E,
\\[10pt]
\df{d L}{dt} = \widetilde{\nu}_E E - (\nu_L + \mu_L) L,
\\[10pt]
\df{d M}{dt} = (1-r) \nu_L L - \mu_M M,
\\[10pt]
\df{dF}{dt} = r \nu_L L (1 - e^{-\beta (M + \gamma_i M_i)}) \f{M}{M+\gamma_i M_i} - \mu_F F,
\\[10pt]
\df{d F_{st}}{dt} = r \nu_L L \big(  e^{-\beta (M + \gamma_i M_i)} + \f{\gamma_i M_i}{M+\gamma_i M_i}(1 - e^{-\beta (M+ \gamma_i M_i)}) \big) - \mu_F F_{st}.
\label{sys:complet}
\eepa
\eeq

Dynamics of the full system \eqref{sys:complet} is not different from that of the following simplified, three-populations system.
We only keep egg, fertile and sterilizing male, and fertile female populations. The value of the hatching parameter $\nu_E$ must be updated to take into account survivorship and development time in the larval stage.
 \beq
 \bepa
 \df{d E}{dt} = b F ( 1 - \f{E}{K}) - (\nu_E + \mu_E) E,
\\[10pt]
\df{d M}{dt} = (1-r) \nu_E E - \mu_M M,
\\[10pt]
\df{dF}{dt} = r \nu_E E (1 - e^{-\beta (M + \gamma_i M_i)}) \f{M}{M+\gamma_i M_i} - \mu_F F.
 \label{sys:simplifie}
 \eepa
 \eeq

The following straightforward lemma means that \eqref{sys:complet} and \eqref{sys:simplifie} are well-suited for population dynamics modeling since all populations, in these systems, remain positive and bounded.
\begin{lemma}
Let $M_i$ be a non-negative, piecewise continuous function on $\RR_+$. The solution to the Cauchy problems associated with \eqref{sys:complet}, \eqref{sys:simplifie} and non-negative initial data is unique, exists on $\RR_+$, is continuous and piecewise continuously differentiable. This solution is also forward-bounded and remains non-negative. It is positive for all positive times if $F(0)$ or $E(0)$ (or also $L(0)$ in the case of \eqref{sys:complet}) is positive.
\end{lemma}
In addition, these systems are monotone in the sense of the monotone systems theory (see \cite{Smith1995}).
\begin{lemma}
The system \eqref{sys:simplifie} is monotone on the set $\mathcal{E}_3 := \{ E \leq K\} \subset \RR_+^3$ for the order induced by $\RR_+^3$ and the restriction of system \eqref{sys:complet} to the four first coordinates (omitting $F_{st}$, which does not appear in any other compartment) is monotone on the set $\mathcal{E}_4 := \{ E \leq K \} \subset \RR_+^4$ for the order induced by $\RR_+^4$.

Moreover, $\mathcal{E}_3$ (respectively $\mathcal{E}_4$) is forward invariant for \eqref{sys:simplifie} (respectively for the restriction of~\eqref{sys:complet} to the four first coordinates), and any trajectory enters it in finite time.
\label{lem:monotone}
\end{lemma}
\begin{proof}
We compute the Jacobian matrix of the system \eqref{sys:simplifie}:
\[
J = \begin{pmatrix}
- \frac{b F}{K} - (\nu_E + \mu_E) & 0 & b (1 - \f{E}{K}) \\
(1-r) \nu_E & - \mu_M & 0 \\
r \nu_E (1 - e^{-\beta (M + \gamma_i M_i)}) \f{M}{M+\gamma_i M_i} & \f{r \nu_E E}{M+\gamma_i M_i} \big( \beta M e^{-\beta (M + \gamma_i M_i)} +(1 - e^{-\beta (M+ \gamma_i M_i)})\f{\gamma_i M_i}{M+\gamma_i M_i}  \big)& - \mu_F
\end{pmatrix}.
\]
It has non-negative extra-diagonal coefficients on $\mathcal{E}_3$, which proves that the system is indeed monotone on this set. In addition, if $E(t_0) > K$ then let $T[t_0] := \{ t \geq t_0, \, \forall t' \in [t_0, t), E(t') > K \} \subset \RR$. Let $T^+ [t_0] := \sup T[t_0]$. For any $t \in T[t_0]$ we have $\dot{E} (t) \leq -(\nu_E + \mu_E) E(t)$. Hence by integration we find that $T^+ [t_0] \leq t_0 + \f{1}{\nu_E + \mu_E} \log(K/E(t_0)) < +\infty$, which proves Lemma \ref{lem:monotone} (the proof being similar for the claims on~\eqref{sys:complet}).
\end{proof}
\begin{remark}
The Allee effect term $1 - \exp(-\beta M)$ can also be interpreted in the light of \cite{Dur.Importance}. This is the probability that an emerging female finds a male to mate with in her neighborhood.

Using a "mean-field'' model of ordinary differential equations here is certainly debatable, since in the case of population extinction the individuals may eventually be very dispersed, and heterogeneity would play a very important role.
However, we think that getting a neat mathematical understanding of the simplest system we study here is a necessary first step before moving to more complex systems. The Allee term compensates, as far as the qualitative behavior is concerned, what the model structurally lacks.
Here, we are able to perform proofs and analytical computations. This gives a starting point for benchmarking what to expect as an output of release programs using sterilizing males, according to the models. 
\label{rem:Allee}
\end{remark}

\subsection{Parameter estimation from experimental data}
\begin{table}[ht!]
\centering
\begin{adjustbox}{max width=\textwidth}
\begin{tabular}{|c l c r|}
\hline
Symbol & Name & Value interval & Source \\
\hline
$r_{\text{viable}}$ & Proportion of viable eggs & $95 - 99$\% & Field collection, \cite[p. 121]{HapairaiPhD} \\
\hline
$N_{\text{eggs}}$ & Number of eggs laid per laying & $55 - 75$ & \cite{RivierePhD} \\
\hline
$\tau_{\text{gono}}$ & Duration of gonotrophic cycle & $4 - 7$ days & \cite{jachowski1954,suzuki1978,RivierePhD}\\
\hline
$\tau_E$ & Egg half-life & $15 - 30$ days & Estimation (to be determined)\\
\hline
$\tau_L$ & Time from hatching to emergence & $8 - 11$ days & Lab data, \cite[p. 104]{HapairaiPhD}\\
\hline
$r_L$ & Survivorship from larva first instar to pupa & $67 - 69$\% & Lab data, \cite[p. 106]{HapairaiPhD}\\
\hline
$r$ & Sex ratio (male:female) & $49$\% & Production data (ILM) \\
\hline
$\tau_M$ & Adult male half-life  & $5-9$ days & Lab data, \cite[p. 50]{HapairaiPhD} \\
\hline
$\tau_F$ & Adult female half-life & $15-21$ days & Lab data, \cite[p. 50]{HapairaiPhD} \\
\hline
$\gamma_i$ & Mating competitiveness of sterilizing males & $1$ & Lab \cite[pp. 51--53]{HapairaiPhD}, field \cite{OConnor2012} \\
\hline
\end{tabular}
\end{adjustbox}
\caption{Parameter values for some populations of {\itshape Aedes polynesiensis} in French Polynesia at a temperature of $27^{\circ} C$.}
\label{table:bioparam}
\end{table}

\begin{table}[ht!]
\centering
\begin{tabular}{|c l c r|}
\hline
Symbol & Name & Formula & Value interval\\
\hline
$b$ & Effective fecundity & $\df{r_{\text{viable}}N_{\text{eggs}} }{\tau_{\text{gono}}}$ & $7.46 - 14.85$ \\
\hline
$\mu_L$ & Larva death rate & $\df{- \log (r_L)}{\tau_L}$ & $0.034 - 0.05$\\
\hline
$\nu_L$ & Larva to adult transition rate & $\df{1}{\tau_L}$ & $0.09 - 0.125$ \\
\hline
$\df{\nu_E}{\widetilde{\nu}_E}$ & Larval coefficient for effective hatching rate & $\df{\nu_L}{\nu_L + \mu_L}$ & $0.64 - 0.79$ \\
\hline
$\mu_E$ & Egg death rate & $\df{\log(2)}{\tau_E}$ & $0.023 - 0.046$\\
\hline
$\mu_M$ & Adult male death rate & $\df{\log(2)}{\tau_M}$ & $0.077 - 0.139$\\
\hline
$\mu_F$ & Adult female death rate & $\df{\log(2)}{\tau_F}$ & $0.033 - 0.046$\\
\hline
\end{tabular}
\caption{Conversion of the biological parameter from Table \ref{table:bioparam} into mathematical parameters for systems \eqref{sys:complet} and \eqref{sys:simplifie}}
\label{table:mathparam}
\end{table}

For numerical simulations, we use experimental (lab and field) values of the biological parameters in \eqref{sys:complet}-\eqref{sys:simplifie}. We consider specifically a population of {\itshape Aedes polynesiensis} in French Polynesia which has been studied in \cite{jachowski1954,suzuki1978,RivierePhD}, and more recently in \cite{Cha.Male,Hap.Effect,hapairai2013population,HapairaiPhD}.

Values of most parameters are given in Table \ref{table:mathparam}, and are deduced from experimental data gathered in Table \ref{table:bioparam}. Some data come from unpublished results obtained at Institut Louis Malard\'{e} during the rearing of {\itshape Aedes polynesiensis} for a pilot IIT program. They are labelled as ``Production data (ILM)''.
Note that we do not give values for $\beta$ and $\widetilde{\nu}_E$ because they are very hard to estimate. Ongoing experiments of one of the author may help approximating them in the future for this {\itshape Aedes polynesiensis} population. 
Finally when it exists, we use the knowledge about population size (male and female) granted by mark-release-recapture experiments to adjust the environmental carrying capacity $K$ for population and season.

\section{Theoretical study of the simplified model}
\label{sec:theory}

For later use, we introduce the usual relations $\ll$, $<$ and $\leq$ on $\RR^d$ (where $d \geq 1$) as the coordinate-wise partial orders on $\RR^d$ induced by the cone $\RR_+^d$. More precisely, for $x, y \in \RR^d$,
\begin{itemize}
\item $x \leq y$ if and only if for all $1 \leq i \leq d$, $x_i \leq y_i$,
\item $x < y$ if and only if $x \leq y$ and $x \not= y$,
\item $x \ll y$ if and only if for all $1 \leq i \leq d$, $x_i < y_i$.
\end{itemize}

\subsection{Constant incompatible male density}
First we study system \eqref{sys:simplifie} with constant incompatible male density $M_i (t) \equiv M_i$.

We introduce the three scalars
\beq
\mathcal{N} := \df{b r \nu_E}{\mu_F (\nu_E + \mu_E)}, \quad \lambda := \f{\mu_M }{(1-r) \nu_E K}, \quad \psi := \f{\lambda}{\beta}
\eeq
and define the function $f : \RR_+^2 \to \RR$, with the two parameters $\mathcal{N}$ and $\psi$:
\beq
f(x, y) := x (1-\psi x) (1 - e^{-(x + y)}) - \f{1}{\mathcal{N}} (x + y).
\label{fun:f}
\eeq

The two aggregated numbers, $\mathcal{N}$ and $\psi$ essentially contain all the information about system~\eqref{sys:simplifie}: $\mathcal{N}$ is the classical basic offspring number, $\psi$ is the ratio between the typical male population size at which the Allee effect comes into play and the male population size at wild equilibrium, as prescribed by the egg carrying capacity.

The ODE system \eqref{sys:simplifie} has simple dynamical properties because it is monotone and we can count its steady states and even know their local stability.
Let $M_i \geq 0$. It is straightforward to show that system \eqref{sys:simplifie} always admits a trivial steady-state $(0,0,0)$ and eventually one (at least) non-trivial steady state $(E^*, M^*, F^*) \in \RR_+^3$ solution of
\[
E = \f{b}{\nu_E + \mu_E} F (1 - \f{E}{K}), \quad
E = \f{\mu_M}{(1-r)\nu_E} M, \quad
F = \f{r \nu_E}{\mu_F} E (1 - e^{-\beta (M+\gamma_i M_i)}) \f{M}{M+\gamma_i M_i}.
\]
Using the first two equation into the third one yields
\[
\f{\mu_F (\nu_E + \mu_E)}{b r \nu_E} (M + \gamma_i M_i) = M (1 - \f{\mu_M}{(1-r) \nu_E K} M)(1 - e^{-\beta (M + \gamma_i M_i)}),
\]
from which we deduce
\[
\left\{
\begin{array}{l}
E^* = K \lambda M^*, \\ 
F^*  = \df{K (\nu_E + \mu_E)}{b} \df{\lambda M^*}{1 - \lambda M^*}, \\
f(\beta M^*, \gamma_i \beta M_i) = 0.
\end{array}
\right.
\]
Hence for a given value $M_i \geq 0$, the number of steady states of \eqref{sys:simplifie} is equal to the number of positive solutions $M^*$ to $f(\beta M^*, \beta \gamma_i M_i) = 0$, plus $1$. The trivial steady state $(0, 0, 0)$ is also locally asymptotically stable (LAS). The following lemma give us additional informations about the positive steady state(s):
\begin{lemma}
Assume $\mathcal{N} > 4 \psi$. Let $\theta_0 \in (0, 1)$ be the unique solution to $1 - \theta_0 = - \f{4 \psi}{\mathcal{N} }\log(\theta_0)$, and
\[
M_i^{\text{crit}} := \frac{1}{\gamma_i \beta} \max_{\theta \in [\theta_0, 1]} \Big( -\log(\theta) - \f{1}{2 \psi} \big( 1 - \sqrt{1 + \f{4 \psi}{\mathcal{N}}\f{\log(\theta)}{1-\theta}} \big) \Big).
\]
If $M_i^{\text{crit}} > 0$ then \eqref{sys:simplifie} has:
\begin{itemize}
\item $0$ positive steady state if $M_i > M_i^{\text{crit}}$,
\item $2$ positive steady states $\EE_- \ll \EE_+$ if $M_i \in [0, M_i^{\text{crit}})$,
\item $1$ positive steady state $\EE$ if $M_i = M_i^{\text{crit}}$.
\end{itemize}
In addition, $\EE_-$ is unstable and $\EE_+$ is locally asymptotically stable.
If $M_i^{\text{crit}} < 0$ then \eqref{sys:simplifie} has no positive steady state, and if $M_i^{\text{crit}} = 0$ then there exists a unique positive steady state. In particular, if $\mathcal{N} \leq 1$ then $M_i^{\text{crit}} < 0$.

On the contrary, if $\mathcal{N} \leq 4 \psi$ then there is no positive steady state.
\label{lem:steadystates}
\end{lemma}
\begin{proof}
Let us give a quick overview of the remainder of the proof, which is detailed in Appendix \ref{app:Lemproof}, page \pageref{app:Lemproof}.
We are going to study in details the solutions $(x, y)$ to $f(x, y) = 0$.
First, we prove that $x < 1 / \psi$. Then, we check that for any $y > 0$, $x \mapsto f(x, y)$ is either concave or convex-concave. In addition, it is straightforward that $f(0, y) < 0$ and $\lim_{x \to +\infty} f(x, y) = - \infty$, so that for any $y > 0$, we conclude that there are either $0$, $1$ or $2$ real numbers $x > 0$ such that $f(x, y) = 0$.

Then, we introduce $\xi = 4 \psi / \mathcal{N}$. 
In fact, in order to determine $(x, y) \in \RR_+^2$ such that $f(x, y) = 0$ we can introduce $\theta = e^{-(x + y)}$ and then check easily that $y = h_{\pm} (\theta)$, where
\beq
h_{\pm} (\theta) =  - \log(\theta) - \f{1}{2 \psi} \pm \f{1}{2 \psi} \sqrt{1 + \xi \f{\log(\theta)}{1-\theta}}.
\label{fun:hpm}
\eeq

Let $\theta_0 (\xi)$ be the unique solution in $(0, 1)$ to $1 - \theta_0 (\xi) = - \xi \log(\theta_0 (\xi))$, and
\beq
\alpha^{\text{crit}} (\xi, \mathcal{N}) := \max_{\theta \in [\theta_0 (\xi), 1]} - \log(\theta) - \f{1}{2 \psi} \big( 1 - \sqrt{1 + \xi \f{\log(\theta)}{1 - \theta}} \big).
\eeq
Collecting the previous facts, and studying the function $h_{\pm}$ (see Appendix \ref{fonctionh}, page \pageref{fonctionh}), we can prove that the next point of Lemma \ref{lem:steadystates} holds with the threshold $M_i^{\text{crit}} = \f{\mathcal{N}}{4 \psi \beta \gamma_i} \alpha^{\text{crit}} (\xi, \mathcal{N})$.

We remark that if $\mathcal{N} \leq 1$ then it is easily checked that $M_i^{\text{crit}} < 0$, using the fact that if $\alpha \in (0, 1)$ then $\sqrt{1 - \alpha} \leq (1 - \alpha)/2$. 
If $\theta \in (\theta_0, 1)$ then $\frac{4 \psi \log(\theta)}{\mathcal{N} (1 - \theta)} < 1$, and therefore
\[
- \log(\theta) - \frac{1}{4\psi} \big(1 - \sqrt{1 + \frac{4 \psi}{\mathcal{N}}\frac{\log(\theta)}{1-\theta}}\big) \leq - \log(\theta) \big( 1 - \frac{1}{\mathcal{N}} \big) - \frac{1}{4 \psi} < 0.
\]

In the final part of the proof, we show that $0$ is always locally stable and then treat separately the cases $M_i = 0$ and $M_i > 0$, showing that, when they exist, the greater positive steady state is locally stable while the smaller one is unstable.

\end{proof}

\begin{remark}
In Lemma \ref{lem:steadystates}, the condition to have at least one positive equilibrium, $\mathcal{N} > 4 \psi$, is very interesting and particularly makes sense when rewritten as $\dfrac{\mathcal{N}}{\lambda} > \dfrac{4}{\beta}$. Indeed $\dfrac{\mathcal{N}}{\lambda}$ can be seen as the theoretical male progeny at next generation, starting from wild equilibrium. If this amount is large enough (larger than some constant times the population size at which the Allee effect comes into play) then the population can maintain. In any case, if this condition is not satisfied, then the population collapses. For the population to maintain: either the fitness is good and thus $\mathcal{N}$ is very large, or the probability of one female to mate is high and thus $1/\beta$ is small. However, whatever the values taken by $\mathcal{N}$ and $\beta$, if, for any reason, the male population at equilibrium decays, the population can be controlled and possibly collapses.
\end{remark}

\begin{remark}
If $\beta$ is not too small, then the ``wild'' steady state is approximately given by $M^*(M_i = 0) \simeq \f{1}{\lambda} (1 - \f{1}{\mathcal{N}})$ and the critical sterilizing level  is approximately $M_i^{\text{crit}} \simeq \widetilde{y} = \f{\mathcal{N}}{4 \lambda \gamma_i} \big( 1 - \f{1}{\mathcal{N}} \big)^ 2$ (see the definition in Appendix \ref{app:Lemproof}, in particular we know that $M_i^{\text{crit}} \leq \widetilde{y}$). As a consequence, the target minimal constant density of sterilizing males compared to wild males in order to get unconditional extinction ({\it i.e.} to make $(0, 0 ,0)$ globally asymptotically stable, see Proposition \ref{prop:almostall}, page \pageref{prop:almostall}) is well approximated by the simple formula
\[
\rho^* := \f{M_i^{\text{crit}}}{M^* (M_i = 0)} \simeq \f{\mathcal{N}-1}{4 \gamma_i}.
\]
With the values from Tables \ref{table:bioparam} and \ref{table:mathparam}, for $\gamma_i = 1$ (this means that introduced male are as competitive as wild ones for mating with wild females), we find
\[
\rho^* \in \Big( \df{7.46 \cdot 0.46 \cdot \nu_E}{4 \cdot 0.046 \cdot (\nu_E + 0.046)}- 0.25 , \df{14.85 \cdot 0.48 \cdot \nu_E}{4 \cdot 0.033 \cdot (\nu_E + 0.023)}-0.25 \Big)
\]
For instance, if $\nu_E = 0.01$ then this interval is $(3.5, 22,7)$, if $\nu_E = 0.05$ then this interval is $(10.6, 51.7)$ and if $\nu_E = 0.1$ then this interval is $(14.1, 61.4)$. As $\nu_E$ goes to $+\infty$, the interval goes to $(20.7, 75.7)$.
\label{rem:ratio}
This example agrees with standard SIT Protocol that indicates to release at least $10$ times more sterile males than wild males, recalling that here we deal with a highly reproductive species (with the above values, the lowest estimated basic reproduction number is $14.9$, obtained for $\nu_E = 0.01$).
\end{remark}

Asymptotic dynamics are easily deduced from the characterization of steady states and local behavior of the system (Lemma \ref{lem:steadystates}), because of the monotonicity (see \cite{Smith1995}).

\begin{proposition}
If \eqref{sys:simplifie} has only the steady state $(0, 0, 0)$ then it is globally asymptotically stable.

If there are two other steady states $\Em \ll \Ep$ then almost every orbit converges to $\Ep$ or $(0, 0, 0)$. 
Let $K_+ := [(0, 0, 0), \Ep]$. The compact set $K_+$ is globally attractive and positively invariant.
The basin of attraction of $(0,0,0)$ contains $[0, \Em)$ and the basin of attraction of $\Ep$ contains $(\Em, \mathbf{\infty})$.
\label{prop:almostall}
\end{proposition}

Now that we have established that the system is typically bistable, the main object to investigate is the separatrix between the two basins of attraction. This is the aim of the next proposition.
\begin{proposition}
Assume $M_i^{\text{crit}} > 0$ and $M_i \in [0, M_i^{\text{crit}})$.

Then there exists a separatrix $\Sigma \subset \RR_+^3$, which is a sub-manifold of dimension $2$, such that for all $X \not= Y \in \Sigma$, $X \not\leq Y$ and $Y \not\leq X$, and for all $\widehat{X} \in \Sigma$, $X_0 > \widehat{X}$ implies that $X(t)$ converges to $\EE_+$, and $X_0 < \widehat{X}$ implies that $X(t)$ converges to $\0$.
In particular, $\EE_- \in \Sigma$.

Let $\Sigma_+ := \big\{ X \in \RR_+^3, \, \exists \widehat{X} \in \Sigma, X > \widehat{X} \big\}$ and $\Sigma_- := \big\{ X \in \RR_+^3, \, \exists \widehat{X} \in \Sigma, X < \widehat{X} \big\}$.
Then $\RR_+^3 = \Sigma_- \cup \Sigma \cup \Sigma_+$, $\Sigma_+$ is the basin of attraction of $\EE_+$ and $\Sigma_-$ is the basin of attraction of $\0$.

In addition, there exists $E_M, F_M > 0$ such that
\[
\Sigma_- \subset \big\{ X \in \RR_+^3, \quad X_1 \leq E_M, \, X_3 \leq F_M \big\}.
\]

\label{prop:separatrice}
\end{proposition}
\begin{remark}
In order to reach extinction, the last point of Proposition \ref{prop:separatrice} states that both egg and fertile female populations must stand simultaneously below given thresholds. This obvious fact receives here a mathematical quantification. With simple words: no matter how low the fertile female population $F$ has dropped, if there remains at least $E_M$ eggs then the wild population will recover.
\end{remark}
\begin{proof}[Proposition \ref{prop:separatrice}]
We state a preliminary fact:
For all $v^0 \in \{ v \in \RR_+^3, \, \forall i, v_i > 0, \, \sum_i v_i = 1 \} =: \mathcal{S}_+^2$, there exists a unique $\rho_0 (v^0)$ such that the solution to \eqref{sys:simplifie} with initial data $\rho v^0$ converges to $\0$ if $\rho < \rho_0 (v^0)$ and to $\EE_+$ if $\rho > \rho_0 (v^0)$.

This fact comes from the strict monotonicity of the system, and from the estimate $\rho_0 (v^0) \leq \max_i \frac{v^0_i}{(\EE_-)_i} < +\infty$, combined with Proposition \ref{prop:almostall}.

Then we claim that $\Sigma = \{ \rho_0 (v^0) v^0, \quad v^0 \in \mathcal{S}^2_+ \}$. The direct inclusion is a corollary of the previous fact. The converse follows from the fact that $\Sigma_\pm$, being the basins of attraction of attracting points, are open sets.

The remainder of the proof consists of a simple computation showing that if $F_0$ or $E_0$ is large enough then for some $t > 0$ we have $(E,M,F) (t) > \EE_-$.
In details, we can prove that if $F_0$ is large enough then for any $E_0, M_0$ and $\epsilon > 0$, we can get $E(s) \geq (1-\epsilon) K$ for $s \in (t_0(\epsilon, E_0, F_0), t_1(\epsilon, E_0, F_0))$, where $t_0$ is decreasing in $F_0$ and $t_1$ is increasing in $F_0$ and unbounded as $F_0$ goes to $+\infty$. Then, if $E > (1-\epsilon) K$ for $\epsilon$ small enough on a large enough time-interval, we deduce $M(t) > (1 - \epsilon)^2 (1 - r) \frac{\nu_E}{\mu_M} K$ for some $t > 0$. Upon choosing $\epsilon$ small enough and $F_0$ large enough we finally get $(E,M,F) (t) > \EE_-$.
The scheme is similar when taking $E_0$ large enough.
\end{proof}

At this stage, we know that starting from the positive equilibrium, and assuming that the population of sterile males $M_i$ is greater than  $M_i^{\text{crit}}$, the solution will reach the basin of attraction of the trivial equilibrium in a finite time, $\tau(M_i)$.
We obtain now quantitative estimates on the duration of this transitory regime. Rigorously, we define
\begin{multline}
\tau(M_i) := \inf \big\{ t \geq 0, \, (E, M, F) (t) \in \Sigma_-(M_i = 0),
\\ 
\text{ where } (E, M, F) (0) = \EE_+ (M_i = 0) \text{ and } (E, M, F) \text{ satisfies \eqref{sys:simplifie}} \big\}.
\end{multline}
We obtain simple upper and lower bounds for $\tau(M_i)$ in terms of various parameters:
\begin{proposition}
   Let $M_i > M_i^{\text{crit}}$, and $Z = Z (\psi)$ be the unique real number in $(0, \f{1}{2\psi})$ such that
   \[
    e^{- Z} = \f{\psi}{1+\psi - \psi Z},
   \]
   and $Z_0 := 1 + \psi - \psi Z$. Then 
   \begin{equation}
    \tau(M_i) \geq \f{1}{\mu_F} \log \big( 1 + \f{\calN^2 (1 - \psi Z)^3}{\psi Z Z_0^2} - \f{\calN (1 - \psi Z)}{\psi Z Z_0} \Big).
    \label{formula:undest}
   \end{equation}
   Let $\sigma = \sgn(\nu_E+\mu_E-\mu_F)$, $\sigma_E := \mu_M / (\nu_E + \mu_E)$ and $\sigma_F := \mu_M / \mu_F$. If $\epsilon := \f{M_+^*}{M_+^*+M_i} < 1/\mathcal{N}$, let 
 \[
 g (\eps) := \sqrt{1 + \f{4 \calN \sigma_E \sigma_F \eps}{(\sigma_F -  \sigma_E)^2}}.
 \]
 Assume that $\sigma_F,\sigma_E > 1$,
 \[
    g(\eps) \sigma (\sigma_F - \sigma_E) < \max \big( (2 \mathcal{N} - 1) \sigma_F + \sigma_E, (2 \sigma_E - 1) \sigma_F \big), \quad (\sigma_F - 1)(\sigma_E - 1) > \eps \calN.
 \]
Then
\begin{multline}
\tau(M_i) \leq \frac{2 \sigma_E}{\mu_F \big(\sigma_F + \sigma_E - g(\epsilon) \sigma (\sigma_F - \sigma_E) \big)} \log \Big(\f{\calN-1}{\psi} \big(\f{(\calN - 1) \sigma_F + 1 - \eps \calN}{(\sigma_F - 1) (\sigma_E - 1) - \eps \calN} 
 \\
 + \f{\sigma_E \sigma_F \big( g(\eps) \sigma (\sigma_F - \sigma_E) + (2 \calN - 1)\sigma_F + \sigma_E  \big)}{\big( 2 \sigma_E \sigma_F - (\sigma_E + \sigma_F) + \sigma (\sigma_F - \sigma_E) g(\eps)\big) g(\eps) \sigma (\sigma_F - \sigma_E)} \big) \Big).
 \label{formula:ovest}
\end{multline}
 \label{prop:upperboundtauMi}
\end{proposition}
\begin{proof}
 The proof relies on explicit computation of sub- and super-solutions, detailed in Appendix~\ref{sec:entrancetime}.
\end{proof}

\begin{remark}
The dependency in $\psi$ of Proposition \ref{prop:upperboundtauMi}'s upper estimate on $\tau$ is approximately equal to $\frac{1}{\min(\nu_E + \mu_E, \mu_F)}$. One order of magnitude of $\psi$ (the ratio between the wild population size and the Allee population size) therefore typically corresponds to the maximum of one adult female and one egg lifespan in terms of release duration needed to get extinction.
\end{remark}

\begin{remark}
 At this stage, we obtain an analytic upper bound only in the case of massive releases ($\epsilon$ small enough). A more refined upper bound could theoretically be obtained, see the derivation in Appendix~\ref{sec:entrancetime}, in particular Lemma~\ref{lem:secondub}.
\end{remark}

\subsection{Adding a control by means of releases}

In a slightly more realistic model, the level of sterilizing male population should vary with time, depending on the releases $t \mapsto u(t) \geq 0$ and on a fixed death rate $\mu_i$. This model reads
\beq
 \bepa
 \df{d E}{dt} = b F ( 1 - \f{E}{K}) - (\nu_E + \mu_E) E,
\\[10pt]
\df{d M}{dt} = (1-r) \nu_E E - \mu_M M,
\\[10pt]
\df{d M_i}{dt} = u(t) - \mu_i M_i,
\\[10pt]
\df{dF}{dt} = r \nu_E E (1 - e^{-\beta (M + \gamma_i M_i)}) \f{M}{M+\gamma_i M_i} - \mu_F F.
 \label{sys:controlled}
 \eepa
 \eeq
In \eqref{sys:controlled}, the number of sterilizing males released between times $t_1$ and $t_2 > t_1$ is simply equal to $\int_{t_1}^{t_2} u(t) dt$.

First, if the release is {\it constant}, say $u(t) \equiv u_0$, then $M_i (t) = e^{-\mu_i t} M_i^0 + \frac{u_0}{\mu_i} (1 - e^{-\mu_i t})$. The special case $M_i^0 = \frac{u_0}{\mu_i}$ leads back to system \eqref{sys:simplifie}, with $M_i \equiv M_i^0$.
For general $M_i^0 \geq 0$, we notice that $M_i (t)$ converges to $\frac{u_0}{\mu_i}$ as $t$ goes to $+\infty$.
\begin{proposition}
Assume $u(t) \equiv u_0$.

If $u_0 > \mu_i M_i^{\text{crit}}$ (defined in Lemma \ref{lem:steadystates}) then $\0$ is globally asymptotically stable.

If $u_0 < \mu_i M_i^{\text{crit}}$, then there exists open sets $\Sigma_- (u_0), \Sigma_+ (u_0) \subset \RR_+^4$, respectively the basins of attraction of $\0$ and $\EE_+$ (defined for \eqref{sys:simplifie} with $M_i = \frac{u_0}{\mu_i}$), separated by a set $\Sigma(u_0)$ which enjoys the same properties as those of $\Sigma(0)$, listed in Proposition \ref{prop:separatrice}.
\label{prop:periodiccontrol}
\end{proposition}
(We do not treat the case $u_0 = \mu_i M_i^{\text{crit}}$).

\begin{proof}
Since system \eqref{sys:controlled} is monotone with respect to the control $u$ (with sign pattern $(-, -, -, +)$), we can use Lemma \ref{lem:steadystates} and Proposition \ref{prop:separatrice} with sub- and super-solution to get this result in a straightforward way.
\end{proof}

From now on we will restrict ourselves to (possibly truncated) time-periodic controls, which means that we assume that there exists $N_r \in \mathbb{Z}_+ \cup \{ +\infty \}$ (the number of release periods), a period $T > 0$ and a function $u_0 : [0, T] \to \RR_+$ such that
\begin{equation}
u(t) = \begin{cases} u_0 (t - n T) &\text{ if } n T \leq t < (n+ 1) T \text{ for some } N_r > n \in \mathbb{Z}_+, \\ 0 &\text{ otherwise.}
\end{cases}
\label{def:uperiodique}
\end{equation}
We use the notation $u \equiv [T, u_0, N_r]$ to describe this control $u$.

As before, we can compute in case \eqref{def:uperiodique}
\begin{align*}
M_i(t) &= e^{-\mu_i t} M_i^0 + \int_0^t u(t') e^{-\mu_i (t - t')} dt' 
\\
&= e^{-\mu_i t} \Big( M_i^0 + \frac{e^{\mu_i (\lfloor \frac{t}{T} \rfloor \wedge N_r) T} - 1}{e^{\mu_i T} - 1} \int_0^T u_0 (t') e^{\mu_i t'} dt' + \int_{T (\lfloor \frac{t}{T} \rfloor \wedge N_r)}^t u(t') e^{\mu_i t'} dt' \Big)
\end{align*}
(Here, for $a, b \in \mathbb{Z}$, we let $a \wedge b = \min(a, b)$).

If $N_r = +\infty$, for any $u_0 \not= 0$ there exists a unique periodic solution $M_i$, uniquely defined by its initial value
\[
M_i^{0, \text{per}} = \frac{1}{1 - e^{-\mu_i T}} \int_0^T u_0 (t') e^{\mu_i t'} dt',
\]
and which we denote by $M_i^{\text{per}} [u_0]$.
\begin{lemma}
Solutions to \eqref{sys:controlled} with $u \equiv [T, u_0, +\infty]$ are such that $M_i$ converges to $M_i^{\text{per}} [u_0]$, and the other compartments converge to a solution of
\beq
 \bepa
 \df{d E}{dt} = b F ( 1 - \f{E}{K}) - (\nu_E + \mu_E) E,
\\[10pt]
\df{d M}{dt} = (1-r) \nu_E E - \mu_M M,
\\[10pt]
\df{dF}{dt} = r \nu_E E (1 - e^{-\beta (M + \gamma_i M_i^{\text{per}}[u_0])}) \f{M}{M+\gamma_i M_i^{\text{per}}[u_0]} - \mu_F F.
 \label{sys:limper}
 \eepa
 \eeq
 Convergence takes place in the sense that the $L^{\infty}$ norm on $(t, +\infty)$ of the difference converges to $0$ as $t$ goes to~$+\infty$.
\end{lemma}
\begin{proof}
 Convergence of $M_i$ is direct from the previous formula. Then, as for Proposition \ref{prop:periodiccontrol} the monotonicity of the system implies the convergence.
\end{proof}

Let $\overline{M}_i [u_0] := \max M_i^{\text{per}}[u_0]$ and $\underline{M}_i [u_0] := \min M_i^{\text{per}}[u_0]$.
\begin{proposition}
 If $\underline{M}_i[u_0] > M_i^{\text{crit}}$ then $\0$ is globally asymptotically stable for \eqref{sys:limper}.
 
 On the contrary, if $\overline{M}_i[u_0] < M_i^{\text{crit}}$ then \eqref{sys:limper} has at least one positive periodic orbit.
 In this case the basin of attraction of $\0$ contains the interval $\big( \0, \EE_- (M_i = \underline{M}_i [u_0]) \big)$, and any initial data above $\EE_+(M_i = \underline{M}_i [u_0])$ converges to $\overline{X}^{\text{per}}[u_0]$.
\end{proposition}
\begin{proof}
System \eqref{sys:limper} is a periodic monotone dynamical system. It
admits a unique non-negative solution $X=\left(E,M,F\right)$.
In fact, we consider the constant sterile population model
\begin{equation}
\bepa
\df{dE_m}{dt}=bF_m\left(1-\df{E_m}{K}\right)-(\nu_{E}+\mu_{E})E_m,
\\[10pt]
\df{dM_m}{dt}=(1-r)\nu_{E}E_m-\mu_{M}M_m,
\\[10pt]
\df{dF_m}{dt}=r\nu_{E}\df{M_m}{M_m+\underline{M}_i[u_0]}\big(1-e^{-\beta (M_m+\gamma_i \underline{M}_i [u_0] )}\big)E_m-\mu_{F}F_m.
\eepa
\label{eq:min}
\end{equation}
such that, using a comparison principle, the solution $X_m=( E_m,M_m,F_m )$
verifies $X_m \geq X$ for all time $t>0$. Thus if $X_m$
converges to $\mathbf{0}$, so will $X$. The behavior of system \eqref{eq:min}
follows from the results obtained in the previous section. A sufficient condition to have
$\mathbf{0}$ globally asymptotically stable in \eqref{sys:limper} is therefore given by $\underline{M}_i^{\text{per}} > M_i^{\text{crit}}$.

The remainder of the claim is better seen at the level of the discrete dynamical system defined by \eqref{sys:limper}. Periodic orbits are in one-to-one correspondence with the fixed points of the monotone mapping $\Phi[u_0] : \RR_+^3 \to \RR_+^3$ defined as the Poincaré application of \eqref{sys:limper} (mapping an initial data to the solution at time $t = T$).
Now, if $X^* := (E^*, M^*, F^*)$ is the biggest ({\it i.e.} stable) steady state of \eqref{sys:simplifie} at level $M_i = \overline{M}_i[u_0] < M_i^{\text{crit}}$, then for any $(E, M, F) \gg (E^*, M^*, F^*)$ and $M'_i \leq M_i$, writing the right-hand side as $\Psi = (\Psi_1, \Psi_2, \Psi_3)$ we have
\begin{align*}
\Psi_1 (E^*, M, F, M'_i) > 0,
\\[10pt]
\Psi_2 (E, M^*, F, M'_i) > 0,
\\[10pt]
\Psi_3 (E, M, F^*, M'_i) > 0.
\end{align*}
In other words, the interval $\big( X^*, +\infty \big)$ is a positively invariant set.
Therefore, $\Phi[u_0](X^*) > X^*$. Thus the sequence $\big( \Phi[u_0]^k (X^*) \big)_k$ is increasing and bounded in $\RR_+^3$: it must converge to some $\underline{X}^* > X^*$.
The same reasoning (with reversed inequalities) applies with the sequence starting at the stable equilibrium associated with $M_i = \underline{M}_i [u_0]$: it must decrease, and thus converge to some $\overline{X}^* \geq \underline{X}^*$.

By our proof we have shown that the open interval $\big( \EE_+(M_i = \underline{M}_i [u_0]), +\infty \big)$ belongs to the basin of attraction of $\overline{X}^{\text{per}}$, and we can also assert that $\big( \EE_- (M_i = \overline{M}_i [u_0]), \EE_+ (M_i = \overline{M}_i [u_0]) \big)$ belongs to the basin of attraction of $\underline{X}^{\text{per}}$, while as usual $\big( \0, \EE_- (M_i = \underline{M}_i [u_0]) \big)$ is in the basin of attraction of~$\0$.
\end{proof}

By a direct application of the previous results
\begin{lemma}
If $\underline{M}_i [u_0] > M_i^{\text{crit}}$ then the control $u \equiv [T, u_0, n]$ (with $n \in \mathbb{Z}_+$) leads to extinction ({\it i.e.} the solution with initial data $\EE_+$ goes to $0$ as $t$ goes to $+\infty$) as soon as
\begin{equation}
n \geq \frac{\tau(\underline{M}_i [u_0])}{T}.
\label{ineq:nbreleases}
\end{equation}
\label{lem:nbreleases}
\end{lemma}

A special case of \eqref{sys:controlled}-\eqref{sys:limper} is obtained by choosing $u_0 = u_0^{\epsilon} = \frac{\Lambda}{\epsilon} \mathds{1}_{[0, \epsilon]}$ for some $\Lambda > 0$ and letting $\epsilon$ go to $0$. Then there exists a unique limit as $\epsilon$ goes to $0$, which is given by the following impulsive differential system derived from \eqref{sys:controlled}:
\beq
 \bepa
 \df{d E}{dt} = b F ( 1 - \f{E}{K}) - (\nu_E + \mu_E) E,
\\[10pt]
\df{d M}{dt} = (1-r) \nu_E E - \mu_M M,
\\[10pt]
\df{d M_i}{dt} = - \mu_i M_i,
\\[10pt]
M_i (nT^+) = M_i (n T) + \Lambda \text{ for } n \in \mathbb{Z}_+ \text{ with } 0 \leq n < N_r,
\\[10pt]
\df{dF}{dt} = r \nu_E E (1 - e^{-\beta (M + \gamma_i M_i)}) \f{M}{M+\gamma_i M_i} - \mu_F F.
 \label{sys:impulsive}
 \eepa
 \eeq
 
 In \eqref{sys:impulsive}, $M_i$ converges to the periodic solution
 \[
 M_i^{\text{imp}}(t) := \lim_{\epsilon \to 0} M_i^{\text{per}} [u_0^{\epsilon}] =  \frac{\Lambda e^{-\mu_i ( t - \lfloor \frac{t}{T} \rfloor T)}}{1 - e^{-\mu_i T}}
 \]
 We can compute explicitly $\underline{M}_i^{\text{imp}} := \frac{\Lambda e^{-\mu_i T}}{1-e^{-\mu_i T}}$ and $\overline{M}_i^{\text{imp}} := \frac{\Lambda }{1-e^{-\mu_i T}}$, respectively the minimum and the maximum of $M_i^{\text{imp}}$.
 We also define the following periodic monotone system as a special case of~\eqref{sys:limper}:
 \beq
 \bepa
 \df{d E}{dt} = b F ( 1 - \f{E}{K}) - (\nu_E + \mu_E) E,
\\[10pt]
\df{d M}{dt} = (1-r) \nu_E E - \mu_M M,
\\[10pt]
\df{dF}{dt} = r \nu_E E (1 - e^{-\beta (M + \gamma_i M_i^{\text{imp}})}) \f{M}{M+\gamma_i M_i^{\text{imp}}} - \mu_F F.
 \label{sys:periodicMi}
 \eepa
 \eeq
The right-hand side of system \eqref{sys:impulsive} is locally Lipschitz
continuous on $\RR^3$. Thus, using a classic existence theorem (Theorem
1.1, p. 3 in \cite{BP1993} ), there exists $T_e>0$ and a unique solution
defined from $(0,T_e)\rightarrow\mathbb{R}^{3}$. Using standard arguments,
it is straightforward to show that the positive orthant $\mathbb{R}_{+}^{3}$
is an invariant region for system \eqref{sys:impulsive}.

We estimate the (minimum) size of the releases $\Lambda$ and
periodicity $T$, such that the wild population goes to extinction.
\begin{proposition}
Let $\mathcal{S} := \df{\big(1-r\big)\nu_{E}\mathcal{N}}{4\mu_{M} \gamma_i}\big(1-\df{1}{\mathcal{N}}\big)^{2}K$. If 
\begin{equation}
T \leq \frac{1}{\mu_i} \log \big( 1 + \frac{\Lambda}{\mathcal{S}} \big)
\label{cond:0GAST}
\end{equation}
then $\0$ is globally asymptotically stable in \eqref{sys:periodicMi}.
Condition \eqref{cond:0GAST} is equivalent to $\Lambda \geq \mathcal{S} \big( e^{\mu_i T} - 1 \big).$
\label{prop:LambdaS}
\end{proposition}
\begin{proof}
We know (see Appendix \ref{app:Lemproof} and Remark \ref{rem:ratio}) that $M_i^{\text{crit}} \leq \frac{\mathcal{N}}{4 \lambda \gamma_i} \big( 1 - \frac{1}{\mathcal{N}} \big)^2$.
Hence the following is a sufficient condition for global asymptotic stability of $\0$:
\[
\underline{M}_{i}^{\text{imp}} \geq\df{\mathcal{N}}{4\lambda \gamma_i}\left(1-\df{1}{\mathcal{N}}\right)^{2}=\df{\left(1-r\right)\nu_{E}\mathcal{N}}{4\mu_{M} \gamma_i}\left(1-\df{1}{\mathcal{N}}\right)^{2}K.
\]
That is
\[
\df{\Lambda e^{-\mu_{i}\tau}}{1-e^{-\mu_{i}\tau}}\geq\df{\left(1-r\right)\nu_{E}\mathcal{N}}{4\mu_{M} \gamma_i}\left(1-\df{1}{\mathcal{N}}\right)^{2}K,
\]
and the result is proved.
\end{proof}
\begin{remark}
As a continuation of Remark \ref{rem:ratio}, we note that Proposition \ref{prop:LambdaS} gives a very simple estimate for the target ratio of sterilizing males per release over initial wild male population as a function of the period between impulsive releases in the form
\[
\rho(T) := \frac{\Lambda}{M^* (M_i = 0)} \simeq \big( e^{\mu_i T} - 1 \big) \frac{\mathcal{N} - 1}{4 \gamma_i}.
\]
\end{remark}

We can specify Lemma \ref{lem:nbreleases} for impulses and combine it with Proposition \ref{prop:upperboundtauMi} to get a sufficient condition for extinction in the impulsive cases:
\begin{proposition}
The impulsive control of amplitude $\Lambda > 0$ and period $T > 0$ satisfying $\Lambda \geq \mathcal{S} \big( e^{\mu_i T} - 1 \big)$ leads to extinction in $n$ impulses if
\begin{equation}
n \geq \frac{\tau(\underline{M}_i^{\text{imp}})}{T}, \text{ where } \underline{M}_i^{\text{imp}} = \frac{\Lambda e^{-\mu_i T}}{1 - e^{-\mu_iT}}.
\label{ineq:nbimpulses}
\end{equation}
\label{prop:nbimpulses}
\end{proposition}

%
\section{Numerical study}
\label{sec:num}
\subsection{Numerical method and parametrization}

In order to preserve positivity of solutions and comparison principle, we use a nonstandard finite-differences (NSFD) scheme to integrate the differential systems (see for instance \cite{Anguelov2012b} for an overview).

For system \eqref{sys:controlled}, it reads
\beq
\bepa
\df{E^{n+1}-E^{n}}{\Phi(\Delta t)} = b F_S^{n} ( 1 - \dfrac{E^{n+1}}{K}) - (\nu_E + \mu_E) E^{n}, \\[10pt]
\df{M^{n+1}-M^{n}}{\Phi(\Delta t)}=(1-r)\nu_{E}E^{n}-\mu_{M}M^{n},
\\[10pt]
\df{M_i^{n+1}-M_i^{n}}{\Phi(\Delta t)}=-\mu_{i}M_{i}^{n}+u^n,
\\[10pt]
\df{F^{n+1}-F^{n}}{\Phi(\Delta t)}=r\nu_{E}\dfrac{M^{n+1}}{M^{n+1}+M_i^{n+1}}\big(1-e^{-\beta\left(M^{n+1}+M_{i}^{n+1}\right)}\big)E^{n}-\mu_{F}F^{n},
\eepa
\label{sch:NSFD}
\eeq
where $\Delta t$ is the time discretization parameter, $\Phi(\Delta t) = \f{1-e^{-Q \Delta t}}{Q}$, $Q = \max\{\mu_M, \mu_F, \nu_E + \mu_E, \mu_i\}$ and $X^n$ (respectively $u^n$) is the approximation of $X(n \Delta t)$ (respectively $u(n \Delta t)$) for $n \in \mathbb{N}$.

\begin{table}[ht]
\centering
 \begin{tabular}{|l|c|c|c|c|c|c|c|c|c||c|}
 \hline
 Parameter & $\beta$ & $b$ & $r$ & $\mu_E$ & $\nu_E$ & $\mu_F$ & $\mu_M$ & $\gamma_i$ & $\mu_i$ & $\Delta t$
 \\
 \hline
 Value & $10^{-4} - 1$ & $10$ & $0.49$ & $0.03$ & $0.001 - 0.25$ & $0.04$ & $0.1$ & $1$ & $0.12$ & $0.1$
 \\
 \hline  
 \end{tabular}
\caption{Numerical values fixed for the simulations.}
\label{table:numbioparam}
\end{table}

We fix the value of some parameters using the values from Tables \ref{table:bioparam} and \ref{table:mathparam} (see Table \ref{table:numbioparam}). Then, in order to get results relevant for an island of $74 \textrm{ ha}$ with an estimated male population of about $69 \textrm{ ha}^{-1}$, we let $\nu_E$ and $\beta$ vary, and fix $K$ such that
\[
 M_+^* = 69\cdot 74 = 5106,
\]
that is
\[
 K = \frac{5106 \cdot \mu_M}{(1-r)\nu_E (1-\f{1}{\mathcal{N}(1-e^{-\beta \cdot 5106})})}.
\]
Recall that for the choice from Table \ref{table:numbioparam}, page \pageref{table:numbioparam}, we have
\[
 \mathcal{N} = 117.5 \frac{\nu_E}{\nu_E + 0.03}.
\]
\begin{remark}
Thus according to the values taken by $\nu_E$ in Table \ref{table:numbioparam}, page \pageref{table:numbioparam}, we have the following bounds for $\calN$:
\[
29 \leq \mathcal{N}  \leq 105.
\]
The other aggregated value of interest, $\psi = \f{\mu_M}{(1-r)\nu_E \beta K} = \f{\calN - (1 - e^{-\beta M_+^*})}{\calN M_+^* \beta}$, ranges from $1.4 \cdot 10^{-4}$ to $2$, approximately.
\end{remark}

All computations were performed using Python programming language (version 3.6.2). The most costly operation was the separatrix approximation, which needed to be done once for each set of parameter values. We first compute points close to the separatrix (see details in Section \ref{sub:numsep}), starting from a regular triangular mesh with $40$ points on each side, then we reduce the points if any comparable pairs appeared. From these (at most $861$) scattered points we build recursively a comparison tree by selecting the point $P$ which minimizes the distance to all other points, and distributing the remaining points into six subtrees, corresponding to each affine orthant whose vertex is $P$. Each tree was saved using \texttt{pickle} module, and loaded when necessary. This was done to reduce the number of operations for checking if a point is below the separatrix, as this needs to be done several times along each computed trajectory. Indeed, using the fact that two points on the separatrix cannot be related by the partial order, one only needs to investigate $3$ of the $6$ remaining orthants to determine if the candidate point is below any of the scattered points or not.
For any given input of released sterilizing males, the computation of a trajectory ended either when the maximal number of iterations was reached (here, we fixed that value at $3 \cdot 10^5$) or when it was found below the separatrix, using the comparison tree. Trial CPU times (on a laptop computer with Intel\verb!®! Core\verb!™! i5-2410M CPU \verb!@! 2.30GHz x 4 processor) for all these operations are given in Table~\ref{tab:CPU}.

\begin{table}[!ht]
 \centering
 \begin{adjustbox}{max width=\textwidth}
 \begin{tabular}{|l|c|c|c|c|c|c|c|}
  \hline
  Operation & Points & Reduction & Tree building & Save & Load & Full trajectory & Stopped trajectory
  \\ \hline
  CPU time (s) & $267$ & $12$ & $6.8$ & $1.8 \cdot 10^{-3}$ & $1 \cdot 10^{-3}$ & $17$ & $0.25$
  \\ \hline
 \end{tabular}
 \end{adjustbox}
\caption{CPU times for the numerical simulations}
\label{tab:CPU}
\end{table}

\subsection{Equilibria and effort ratio}

We first compute the position of equilibria for a range of values of $\beta$ and $\widetilde{\nu}_E$. This enables us to compute the effort ratio $\rho^*$, defined in Remark \ref{rem:ratio} as the ratio between the wild steady state male population $M^* (M_i = 0)$ and the critical constant value of sterilizing males $M_i^{\text{crit}}$ necessary in order to make $\0$ globally asymptotically stable. Values are shown in Table \ref{table:effratio}. 

\begin{table}[ht]
 \centering
 \begin{tabular}{|l|c|c|c|c|c|c|c|c|c|}
 \hline
 $\nu_E$ & 0.005 & 0.010 & 0.020 & 0.030 & 0.050 & 0.100 & 0.150 & 0.200 & 0.250 \\ \hline
 $\rho^*$ & $16$ & $30$ & $48$ & $60$ & $76$ & $93$ & $101$ & $106$ & $108$ \\ \hline
\end{tabular}
 \caption{Effort ratio $\rho^*=M_i^{\text{crit}} / M^* (M_i = 0)$ for various values of $\nu_E$. For this range of parameters, $\rho^*$ is practically independent on $\beta \in [10^{-4}, 1]$.}
 \label{table:effratio}
\end{table}
We note that $\rho^*$ depends practically only on $\nu_E$, because the Allee (with parameter $\beta$) does not apply at high population levels. In fact the ratio (and thus the control effort) increases with increasing values of $\nu_E$, that favor the maintenance of the wild population (the larger the value of $\nu_E$, the larger the value of $\mathcal{N}$ and the shorter the period in the eggs compartment).

\subsection{Computation of the basin of attraction of \texorpdfstring{$\0$}{0} for \texorpdfstring{\eqref{sys:simplifie}}{the system}}
\label{sub:numsep}
We start from a regular triangular mesh of the triangle $\{ (E,M,F) \in \RR_+^3, \quad E + M + F = 1 \}$, with $40$ points on each side. Given $\epsilon > 0$, for each vertex $V$ of this mesh we compute $\lambda \in (0, +\infty)$ such that $\lambda V \in \Sigma_-$ and $(1+\epsilon) \lambda V \in \Sigma_+$. The points $\lambda V$ (which are numerically at distance at most $\epsilon$ of the separatrix $\Sigma$) are then plotted.

\begin{figure}[ht]
 \includegraphics[width=.5\linewidth]{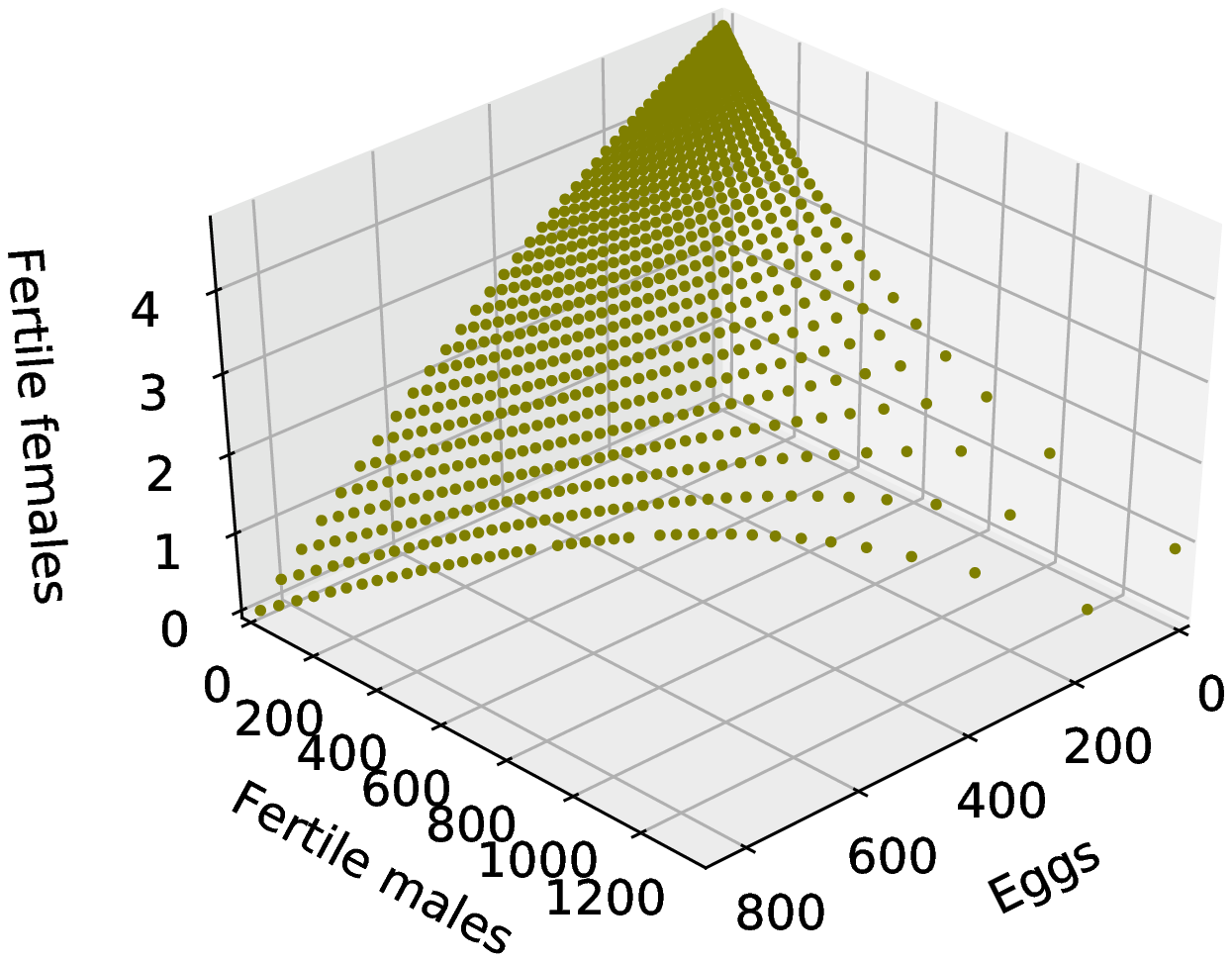}
 \includegraphics[width=.5\linewidth]{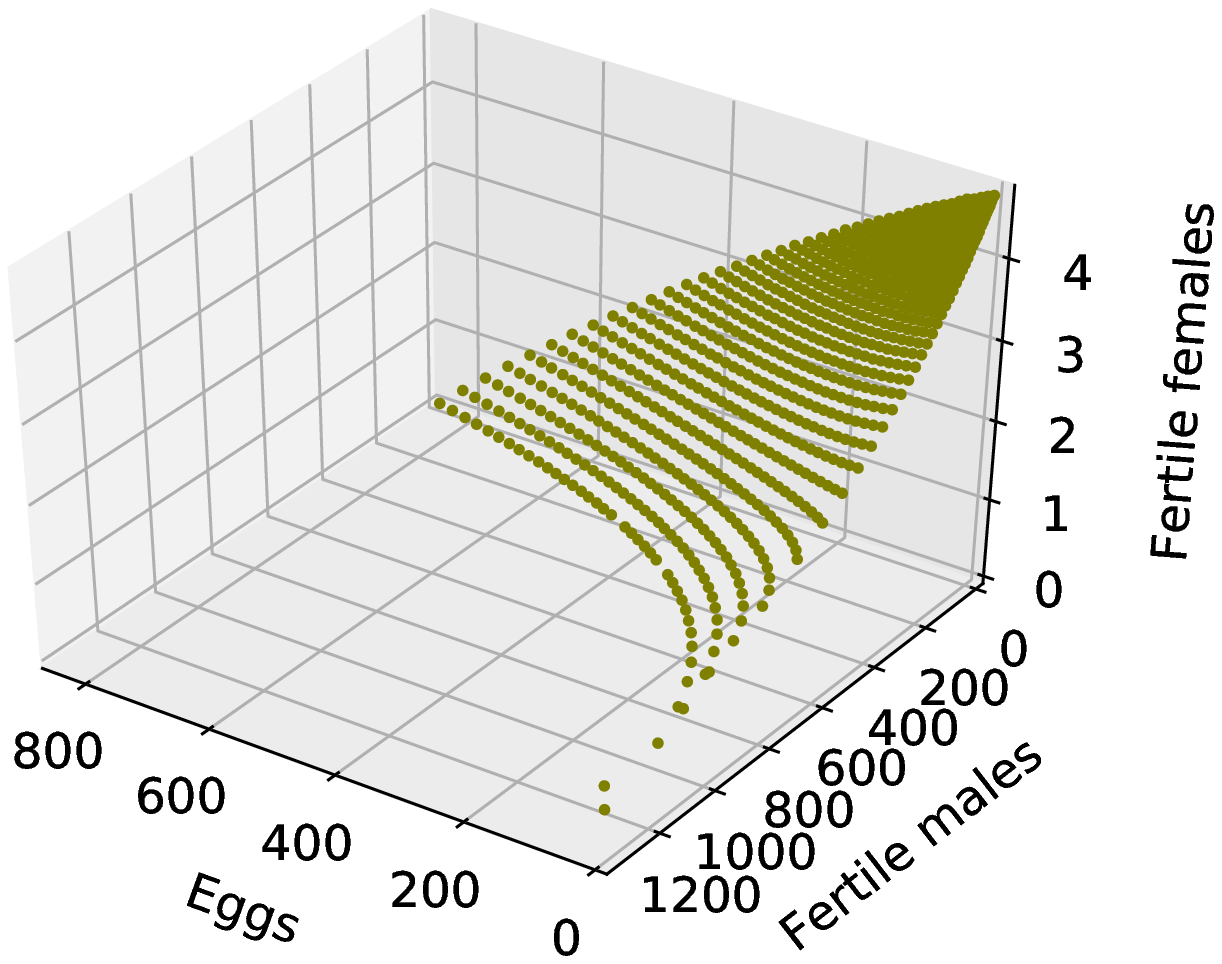}
 \caption{Two viewpoints of scattered points lying around the separatrix ($\eps = 10^{-2}$) for $\nu_E = 0.1$ and $\beta = 10^{-4}$. In this case, $5$ females or $900$ eggs are enough to prevent population elimination.}
 \label{fig:bassin}
\end{figure}

Figure \ref{fig:bassin} is typically the kind of figure that we can draw for each set of parameters. Depending on the parameters values, the basin of attraction of $\0$ can be tiny, or not. Its shape emphasizes the important role of eggs and, even, males abundance in the maintenance of the wild population. In fact, even if almost all females have disappeared, the control must go on in order to further reduce the stock of eggs before eventually reaching the separatrix.

\subsection{Constant releases and entrance time into basin}

For the same set of parameters as before, we compute the entrance time into the basin of $\0$.

First, we use Proposition \ref{prop:upperboundtauMi} to get in Table \ref{table:undest} an underestimation of the entrance time, whatever the releasing effort could be, these entrance times represent the minimal time under which the SIT control cannot be successful (in fact, this under-estimation corresponds to the situation where $M_i = +\infty$, that is an infinite releasing effort).
\begin{table}[ht]
\centering
\begin{adjustbox}{max width=\textwidth}
  \begin{tabular}{|l|c|c|c|c|c|}
\hline 
$\nu_E \backslash \beta$ &$10^{-4}$ &$10^{-3}$ &$10^{-2}$ &$10^{-1}$ &$10^{0}$ \\ \hline 
0.005 &  $  63$ & $ 151$ & $ 204$ & $ 253$ & $ 303$ \\ \hline 
0.010 &  $  93$ & $ 180$ & $ 232$ & $ 281$ & $ 331$ \\ \hline 
0.020 &  $ 118$ & $ 203$ & $ 256$ & $ 304$ & $ 354$ \\ \hline 
0.030 &  $ 130$ & $ 215$ & $ 267$ & $ 315$ & $ 365$ \\ \hline 
0.050 &  $ 141$ & $ 226$ & $ 278$ & $ 327$ & $ 377$ \\ \hline 
0.100 &  $ 152$ & $ 236$ & $ 289$ & $ 337$ & $ 387$ \\ \hline 
0.150 &  $ 156$ & $ 240$ & $ 293$ & $ 341$ & $ 391$ \\ \hline 
0.200 &  $ 158$ & $ 242$ & $ 295$ & $ 343$ & $ 393$ \\ \hline 
0.250 &  $ 160$ & $ 244$ & $ 296$ & $ 344$ & $ 395$ \\ \hline 
\end{tabular}
  \begin{tabular}{|l|c|c|c|c|c|}
\hline 
$10^{-4}$ &$10^{-3}$ &$10^{-2}$ &$10^{-1}$ &$10^{0}$ \\ \hline 
$ 323$ & $ 445$ & $ 571$ & $ 697$ & $ 824$ \\ \hline 
$ 361$ & $ 475$ & $ 592$ & $ 708$ & $ 825$ \\ \hline 
$ 381$ & $ 485$ & $ 590$ & $ 695$ & $ 800$ \\ \hline 
$ 391$ & $ 488$ & $ 587$ & $ 685$ & $ 783$ \\ \hline 
$ 440$ & $ 530$ & $ 621$ & $ 713$ & $ 804$ \\ \hline 
N/A & N/A & N/A & N/A & N/A \\ \hline 
N/A & N/A & N/A & N/A & N/A \\ \hline 
N/A & N/A & N/A & N/A & N/A \\ \hline 
N/A & N/A & N/A & N/A & N/A \\ \hline 
\end{tabular}
  \begin{tabular}{|c|c|c|c|c|}
\hline 
$10^{-4}$ &$10^{-3}$ &$10^{-2}$ &$10^{-1}$ &$10^{0}$ \\ \hline 
$ 258$ & $ 351$ & $ 448$ & $ 545$ & $ 642$ \\ \hline 
$ 286$ & $ 374$ & $ 464$ & $ 553$ & $ 643$ \\ \hline 
$ 301$ & $ 381$ & $ 462$ & $ 544$ & $ 625$ \\ \hline 
$ 307$ & $ 383$ & $ 461$ & $ 538$ & $ 615$ \\ \hline 
$ 332$ & $ 404$ & $ 477$ & $ 550$ & $ 623$ \\ \hline 
N/A & N/A & N/A & N/A & N/A \\ \hline 
N/A & N/A & N/A & N/A & N/A \\ \hline 
N/A & N/A & N/A & N/A & N/A \\ \hline 
N/A & N/A & N/A & N/A & N/A \\ \hline 
\end{tabular}
\end{adjustbox}
 \caption{Left: under-estimation of the entrance time into the basin of $\0$ from the analytic formula~\eqref{formula:undest}. Middle and right: over-estimation of the entrance time into the basin of $\0$ from formula~\eqref{formula:ovest} with $\eps = \frac{M_+^*}{M_+^* + \phi M_i^{\text{crit}}}$, when applicable, for $\phi = 8$ (middle) and $\phi = 4$ (right).}
 \label{table:undest}
\end{table}

Then we compute numerically the entrance time for a range of releasing efforts. In details, computations were performed for $M_i = \phi M_i^{\text{crit}}$ with $\phi \in \{ 1.2, 1.4, 1.6, 1.8, 2, 4, 8\}$. Results are shown in Table~\ref{table:entrancetimephivar} for $\phi = 1.2$, $\phi = 2$ and $\phi = 8$.

\begin{table}[ht]
 \begin{adjustbox}{max width=\textwidth}
  \begin{tabular}{|l|c|c|c|c|c|}
\hline 
$\nu_E \backslash \beta$ &$10^{-4}$ &$10^{-3}$ &$10^{-2}$ &$10^{-1}$ &$10^{0}$ \\ \hline 
0.005 &  $ 168$ & $ 286$ & $ 363$ & $ 435$ & $ 504$ \\ \hline 
0.010 &  $ 200$ & $ 305$ & $ 376$ & $ 441$ & $ 505$ \\ \hline 
0.020 &  $ 219$ & $ 313$ & $ 377$ & $ 437$ & $ 495$ \\ \hline 
0.030 &  $ 225$ & $ 314$ & $ 375$ & $ 434$ & $ 492$ \\ \hline 
0.050 &  $ 228$ & $ 314$ & $ 373$ & $ 431$ & $ 488$ \\ \hline 
0.100 &  $ 231$ & $ 314$ & $ 372$ & $ 430$ & $ 488$ \\ \hline 
0.150 &  $ 232$ & $ 315$ & $ 373$ & $ 431$ & $ 489$ \\ \hline 
0.200 &  $ 233$ & $ 316$ & $ 375$ & $ 433$ & $ 491$ \\ \hline 
0.250 &  $ 234$ & $ 318$ & $ 376$ & $ 434$ & $ 493$ \\ \hline 
\end{tabular}
  \begin{tabular}{|c|c|c|c|c|}
\hline 
$10^{-4}$ &$10^{-3}$ &$10^{-2}$ &$10^{-1}$ &$10^{0}$ \\ \hline 
 $ 148$ & $ 262$ & $ 338$ & $ 409$ & $ 478$ \\ \hline 
 $ 180$ & $ 283$ & $ 352$ & $ 417$ & $ 480$ \\ \hline 
 $ 199$ & $ 292$ & $ 355$ & $ 415$ & $ 473$ \\ \hline 
 $ 207$ & $ 295$ & $ 355$ & $ 413$ & $ 471$ \\ \hline 
 $ 212$ & $ 297$ & $ 355$ & $ 413$ & $ 470$ \\ \hline 
 $ 215$ & $ 298$ & $ 356$ & $ 414$ & $ 472$ \\ \hline 
 $ 217$ & $ 300$ & $ 358$ & $ 416$ & $ 474$ \\ \hline 
 $ 219$ & $ 302$ & $ 360$ & $ 418$ & $ 476$ \\ \hline 
 $ 220$ & $ 303$ & $ 362$ & $ 420$ & $ 478$ \\ \hline 
\end{tabular}
  \begin{tabular}{|c|c|c|c|c|}
\hline 
$10^{-4}$ &$10^{-3}$ &$10^{-2}$ &$10^{-1}$ &$10^{0}$ \\ \hline 
 $ 128$ & $ 237$ & $ 311$ & $ 380$ & $ 449$ \\ \hline 
 $ 160$ & $ 258$ & $ 326$ & $ 391$ & $ 454$ \\ \hline 
 $ 180$ & $ 270$ & $ 333$ & $ 392$ & $ 450$ \\ \hline 
 $ 188$ & $ 274$ & $ 334$ & $ 392$ & $ 450$ \\ \hline 
 $ 194$ & $ 278$ & $ 336$ & $ 394$ & $ 452$ \\ \hline 
 $ 200$ & $ 282$ & $ 340$ & $ 398$ & $ 456$ \\ \hline 
 $ 202$ & $ 285$ & $ 343$ & $ 401$ & $ 459$ \\ \hline 
 $ 205$ & $ 287$ & $ 345$ & $ 403$ & $ 462$ \\ \hline 
 $ 206$ & $ 289$ & $ 347$ & $ 406$ & $ 464$ \\ \hline 
\end{tabular}
 \end{adjustbox}
 \caption{Entrance time into the basin of $\0$ (in days) for various values of $(\nu_E,\beta)$, with $M_i = 1.2 M_i^{\text{crit}}$ (left), $M_i = 2 M_i^{\text{crit}}$ (middle) and $M_i = 8 M_i^{\text{crit}}$ (right).}
 \label{table:entrancetimephivar}
\end{table}

\begin{table}[ht]
 \begin{adjustbox}{max width=\textwidth}
   \begin{tabular}{|l|c|c|c|c|c|}
\hline 
$\nu_E \backslash \beta$ &$10^{-4}$ &$10^{-3}$ &$10^{-2}$ &$10^{-1}$ &$10^{0}$ \\ \hline 
0.005 &  $ 399$ & $ 680$ & $ 863$ & $1034$ & $1199$ \\ \hline 
0.010 &  $ 854$ & $1302$ & $1603$ & $1880$ & $2154$ \\ \hline 
0.020 &  $1513$ & $2166$ & $2603$ & $3022$ & $3423$ \\ \hline 
0.030 &  $1950$ & $2726$ & $3253$ & $3761$ & $4264$ \\ \hline 
0.050 &  $2482$ & $3421$ & $4059$ & $4686$ & $5315$ \\ \hline 
0.100 &  $3100$ & $4218$ & $5000$ & $5777$ & $6553$ \\ \hline 
0.150 &  $3383$ & $4581$ & $5434$ & $6278$ & $7122$ \\ \hline 
0.200 &  $3545$ & $4806$ & $5694$ & $6578$ & $7461$ \\ \hline 
0.250 &  $3649$ & $4956$ & $5869$ & $6779$ & $7689$ \\ \hline 
\end{tabular}
  \begin{tabular}{|c|c|c|c|c|}
\hline 
$10^{-4}$ &$10^{-3}$ &$10^{-2}$ &$10^{-1}$ &$10^{0}$ \\ \hline 
 $ 587$ & $1036$ & $1338$ & $1619$ & $1893$ \\ \hline 
 $1283$ & $2009$ & $2499$ & $2962$ & $3416$ \\ \hline 
 $2296$ & $3367$ & $4092$ & $4782$ & $5452$ \\ \hline 
 $2989$ & $4260$ & $5132$ & $5976$ & $6815$ \\ \hline 
 $3837$ & $5381$ & $6434$ & $7483$ & $8529$ \\ \hline 
 $4817$ & $6675$ & $7975$ & $9268$ & $10563$ \\ \hline 
 $5274$ & $7268$ & $8688$ & $10095$ & $11502$ \\ \hline 
 $5549$ & $7638$ & $9117$ & $10588$ & $12060$ \\ \hline 
 $5717$ & $7884$ & $9405$ & $10922$ & $12438$ \\ \hline 
\end{tabular}
  \begin{tabular}{|c|c|c|c|c|}
\hline 
$10^{-4}$ &$10^{-3}$ &$10^{-2}$ &$10^{-1}$ &$10^{0}$ \\ \hline 
 $2024$ & $3749$ & $4920$ & $6027$ & $7115$ \\ \hline 
 $4548$ & $7343$ & $9257$ & $11112$ & $12912$ \\ \hline 
 $8317$ & $12451$ & $15331$ & $18040$ & $20736$ \\ \hline 
 $10862$ & $15871$ & $19319$ & $22691$ & $26040$ \\ \hline 
 $14058$ & $20188$ & $24395$ & $28588$ & $32774$ \\ \hline 
 $17891$ & $25266$ & $30457$ & $35640$ & $40813$ \\ \hline 
 $19651$ & $27618$ & $33285$ & $38913$ & $44541$ \\ \hline 
 $20757$ & $29073$ & $34979$ & $40865$ & $46750$ \\ \hline 
 $21443$ & $30036$ & $36123$ & $42188$ & $48254$ \\ \hline 
\end{tabular}
 \end{adjustbox}
 \caption{Total effort ratio to get into the basin of $\0$ for various values of $(\nu_E,\beta)$, with $M_i = 1.2 M_i^{\text{crit}}$ (left), $M_i = 2 M_i^{\text{crit}}$ (middle) and $M_i = 8 M_i^{\text{crit}}$ (right). The total effort ratio in this case is defined as $M_i/M^*_+$ multiplied by $\mu_i$ times the entrance time, and corresponds to the number of males that should be released at a constant level, divided by the initial male population.}
 \label{table:totaleffortphivar}
\end{table}

We notice that the entrance times corresponding to the biggest effort ratio are of the same order of magnitude as the analytic under-estimation from formula \eqref{formula:undest}.

Another interesting output of Table \ref{table:entrancetimephivar} is that the release effort ratio is not so important in terms of duration of the control: depending on the values taken by $\nu_E$ and $\beta$, the lowest ratio needs between 4 to 7 more weeks to reach the basin, than the largest ratio. Contrary to what could have been expected, there is no linear relationship. This can be explained by the fact that a female mates only once. Thus if males are in abundance, all females have mated, and then many released males become useless with regards to sterilization. Of course, this has to be mitigated taking into account that our model implicitly assumes a homogeneous distribution, while in real, environmental parameters (like vegetation, climate, etc.) have to be taken into account \cite{dufourd2013}.
Last but not least, Table~\ref{table:totaleffortphivar}, page~\pageref{table:totaleffortphivar}, clearly emphasizes that a large effort ratio, {\itshape i.e.} $\phi=8$, means the use (and then the production) of a large number of sterile males with a really small time-saving compared to the case $\phi=2$. For instance with $\nu_E = 0.05$ and $\beta = 10^{-2}$, the total effort ratio for $\phi = 8$ is approximately $6$ times larger than for $\phi = 2$ ($24395$ against $4059$), with a time-saving of $37$ days, that is approximately one tenth of the total protocol duration ($336$ days against $373$).

In other words, releasing a large number of sterile males is not necessarily a good strategy, from the economical point of view, but also from the control point of view.


In the next subsection, We consider a more realistic scenario, where sterile males are released periodically and instantaneously (system \eqref{sys:periodicMi}).
\subsection{Periodic releases}
In the case of periodic releases by pulses $u = [T, \Lambda \delta_0, \infty]$, for a given couple $(\nu_E, \beta)$ we compute the first time $t > 0$ such that $(E,M,F) (t)$ is below one point of the previously computed separatrix.

We performed the computations with $T \in \{ 1, 2, 3, 4, 5, 6, 7, 8, 9, 10 \}$, choosing 
\[
\Lambda = K \frac{\phi (1 - r) \nu_E \mathcal{N}}{4 \mu_M} \big( 1 - \frac{1}{\mathcal{N}} \big)^2 \big( e^{\mu_i T} - 1 \big) 
\]
for $\phi \in \{ 1.2, 1.4, 1.6, 1.8, 2, 4, 8\}$.

For all combinations of $(\nu_E,\beta)$, we indicate in Table \ref{table:periodic1} the maximal and minimal (with respect to $(T, \phi)$) total effort ratio $\rho_{\text{tot}}$ defined as the number of released mosquitoes at the time when the basin of $\0$ is reached, divided by the initial male population that is:
\[
 \rho_{\text{tot}} := n_{\text{tot}} \Lambda / M_+^*, \quad n_{\text{tot}} = \min \{ \lfloor t/T \rfloor, \, (E,M,F) (t) \in \Sigma_- \}.
\]
These extremal values are obtained for a period $T$ and with an entrance time $t_*$ that are shown in parentheses.
We also indicate in Table \ref{table:periodic2} the maximal and minimal entrance times, obtained for a period $T$ and an effort ratio $\rho_{\text{tot}}$ that are shown in parentheses. Note that consistently, the minimal entrance time is always obtained for $\phi = 8$ and corresponds to the maximal effort ratio. Maximal entrance time is obtained for $T = 1$ (minimal tested period) and the minimal entrance time is obtained for $T = 10$ (maximal tested period). However, the minimal effort ratio is sometimes obtained with $T =2$.

\begin{table}[ht]
 \begin{adjustbox}{max width=\textwidth}
\begin{tabular}{|l|c|c|c|c|c|}
\hline 
$\nu_E \backslash \beta$ &$10^{-4}$ &$10^{-3}$ &$10^{-2}$ &$10^{-1}$ &$10^{0}$ \\ \hline 
0.005 &  $ 282$ ($  2, 287$) & $ 384$ ($  2, 491$) & $ 448$ ($  1, 608$) & $ 502$ ($  1, 682$) & $ 554$ ($  1, 752$) \\ \hline 
0.010 &  $ 547$ ($  1, 344$) & $ 698$ ($  2, 497$) & $ 796$ ($  1, 602$) & $ 884$ ($  1, 669$) & $ 969$ ($  2, 805$) \\ \hline 
0.020 &  $ 900$ ($  1, 357$) & $1112$ ($  1, 519$) & $1253$ ($  1, 585$) & $1386$ ($  1, 647$) & $1504$ ($  2, 771$) \\ \hline 
0.030 &  $1125$ ($  3, 363$) & $1371$ ($  1, 510$) & $1538$ ($  1, 572$) & $1696$ ($  2, 693$) & $1839$ ($  2, 752$) \\ \hline 
0.050 &  $1383$ ($  2, 379$) & $1669$ ($  1, 496$) & $1875$ ($  1, 556$) & $2066$ ($  2, 672$) & $2238$ ($  2, 730$) \\ \hline 
0.100 &  $1655$ ($  2, 370$) & $1997$ ($  1, 480$) & $2238$ ($  1, 539$) & $2458$ ($  2, 650$) & $2678$ ($  2, 708$) \\ \hline 
0.150 &  $1772$ ($  1, 388$) & $2134$ ($  1, 473$) & $2394$ ($  2, 583$) & $2632$ ($  2, 641$) & $2871$ ($  2, 699$) \\ \hline 
0.200 &  $1834$ ($  1, 384$) & $2213$ ($  1, 470$) & $2482$ ($  2, 578$) & $2731$ ($  2, 636$) & $2979$ ($  1, 738$) \\ \hline 
0.250 &  $1873$ ($  1, 382$) & $2263$ ($  1, 468$) & $2531$ ($  2, 575$) & $2787$ ($  2, 633$) & $3043$ ($  2, 692$) \\ \hline 
\end{tabular}
\begin{tabular}{|c|c|c|c|c|}
\hline 
$10^{-4}$ &$10^{-3}$ &$10^{-2}$ &$10^{-1}$ &$10^{0}$ \\ \hline 
 $1095$ ($ 10, 135$) & $1838$ ($ 10, 248$) & $2450$ ($ 10, 323$) & $2986$ ($ 10, 393$) & $3522$ ($ 10, 462$) \\ \hline 
 $2317$ ($ 10, 168$) & $3575$ ($ 10, 268$) & $4536$ ($ 10, 337$) & $5499$ ($ 10, 402$) & $6323$ ($ 10, 466$) \\ \hline 
 $4139$ ($ 10, 188$) & $6015$ ($ 10, 280$) & $7573$ ($ 10, 343$) & $8909$ ($ 10, 402$) & $10246$ ($ 10, 460$) \\ \hline 
 $5448$ ($ 10, 196$) & $7829$ ($ 10, 283$) & $9506$ ($ 10, 343$) & $11183$ ($ 10, 402$) & $12581$ ($ 10, 460$) \\ \hline 
 $7155$ ($ 10, 201$) & $9818$ ($ 10, 286$) & $11921$ ($ 10, 344$) & $14025$ ($ 10, 402$) & $15778$ ($ 10, 460$) \\ \hline 
 $8794$ ($ 10, 206$) & $12114$ ($ 10, 289$) & $14709$ ($ 10, 347$) & $17305$ ($ 10, 405$) & $19900$ ($ 10, 463$) \\ \hline 
 $9522$ ($ 10, 209$) & $13603$ ($ 10, 291$) & $15948$ ($ 10, 350$) & $18762$ ($ 10, 408$) & $21576$ ($ 10, 466$) \\ \hline 
 $10431$ ($ 10, 211$) & $14201$ ($ 10, 293$) & $17138$ ($ 10, 352$) & $19586$ ($ 10, 410$) & $22524$ ($ 10, 468$) \\ \hline 
 $10709$ ($ 10, 212$) & $14584$ ($ 10, 295$) & $17601$ ($ 10, 353$) & $20618$ ($ 10, 412$) & $23133$ ($ 10, 470$) \\ \hline 
\end{tabular}
\end{adjustbox}
 \caption{Minimal (left) and maximal (right) total effort ratio to get into the basin of $\0$ (in days) for various values of $(\nu_E,\beta)$, the minimum and maximum being taken with respect to $(T,\phi)$, with a period and an entrance time shown in parentheses. The total effort ratio is defined as the total number of released male mosquitoes divided by the initial (wild) male mosquito population.}
 \label{table:periodic1}
\end{table}

\begin{table}[ht]
 \begin{adjustbox}{max width=\textwidth}
\begin{tabular}{|l|c|c|c|c|c|}
\hline 
$\nu_E \backslash \beta$ &$10^{-4}$ &$10^{-3}$ &$10^{-2}$ &$10^{-1}$ &$10^{0}$ \\ \hline 
0.005 &  $135$ ($ 10, 1095$) & $248$ ($ 10, 1838$) & $323$ ($ 10, 2450$) & $393$ ($ 10, 2986$) & $462$ ($ 10, 3522$) \\ \hline 
0.010 &  $168$ ($ 10, 2317$) & $268$ ($ 10, 3575$) & $337$ ($ 10, 4536$) & $402$ ($ 10, 5499$) & $466$ ($ 10, 6323$) \\ \hline 
0.020 &  $188$ ($ 10, 4139$) & $280$ ($ 10, 6015$) & $343$ ($ 10, 7573$) & $402$ ($ 10, 8909$) & $460$ ($ 10, 10246$) \\ \hline 
0.030 &  $196$ ($ 10, 5448$) & $283$ ($ 10, 7829$) & $343$ ($ 10, 9506$) & $402$ ($ 10, 11183$) & $460$ ($ 10, 12581$) \\ \hline 
0.050 &  $201$ ($ 10, 7155$) & $286$ ($ 10, 9818$) & $344$ ($ 10, 11921$) & $402$ ($ 10, 14025$) & $460$ ($ 10, 15778$) \\ \hline 
0.100 &  $206$ ($ 10, 8794$) & $289$ ($ 10, 12114$) & $347$ ($ 10, 14709$) & $405$ ($ 10, 17305$) & $463$ ($ 10, 19900$) \\ \hline 
0.150 &  $209$ ($ 10, 9522$) & $291$ ($ 10, 13603$) & $350$ ($ 10, 15948$) & $408$ ($ 10, 18762$) & $466$ ($ 10, 21576$) \\ \hline 
0.200 &  $211$ ($ 10, 10431$) & $293$ ($ 10, 14201$) & $352$ ($ 10, 17138$) & $410$ ($ 10, 19586$) & $468$ ($ 10, 22524$) \\ \hline 
0.250 &  $212$ ($ 10, 10709$) & $295$ ($ 10, 14584$) & $353$ ($ 10, 17601$) & $412$ ($ 10, 20618$) & $470$ ($ 10, 23133$) \\ \hline 
\end{tabular}
\begin{tabular}{|c|c|c|c|c|}
\hline 
$10^{-4}$ &$10^{-3}$ &$10^{-2}$ &$10^{-1}$ &$10^{0}$ \\ \hline 
 $456$ ($  1,  317$) & $667$ ($  1,  420$) & $752$ ($  1,  474$) & $826$ ($  1,  521$) & $896$ ($  1,  565$) \\ \hline 
 $528$ ($  1,  629$) & $661$ ($  1,  749$) & $735$ ($  1,  833$) & $803$ ($  1,  909$) & $868$ ($  1,  982$) \\ \hline 
 $534$ ($  1, 1012$) & $642$ ($  1, 1179$) & $708$ ($  1, 1300$) & $771$ ($  1, 1414$) & $830$ ($  1, 1522$) \\ \hline 
 $527$ ($  1, 1246$) & $627$ ($  1, 1445$) & $690$ ($  1, 1588$) & $749$ ($  1, 1724$) & $807$ ($  1, 1860$) \\ \hline 
 $514$ ($  1, 1513$) & $605$ ($  1, 1749$) & $666$ ($  1, 1925$) & $724$ ($  1, 2090$) & $782$ ($  1, 2257$) \\ \hline 
 $494$ ($  1, 1787$) & $581$ ($  1, 2072$) & $640$ ($  1, 2279$) & $698$ ($  1, 2485$) & $755$ ($  1, 2692$) \\ \hline 
 $483$ ($  1, 1896$) & $569$ ($  1, 2200$) & $628$ ($  1, 2428$) & $686$ ($  1, 2652$) & $744$ ($  1, 2877$) \\ \hline 
 $477$ ($  1, 1953$) & $563$ ($  1, 2272$) & $622$ ($  1, 2510$) & $680$ ($  1, 2745$) & $738$ ($  1, 2979$) \\ \hline 
 $473$ ($  1, 1988$) & $559$ ($  1, 2317$) & $618$ ($  1, 2562$) & $676$ ($  1, 2802$) & $734$ ($  1, 3043$) \\ \hline 
\end{tabular}
 \end{adjustbox}
 \caption{Minimal (left) and maximal (right) entrance time into the basin of $\0$ (in days) for various values of $(\nu_E,\beta)$, the minimum and maximum being taken with respect to $(T,\phi)$, with a period and a total effort ratio shown in parentheses.}
 \label{table:periodic2}
\end{table}

Comparing Tables \ref{table:totaleffortphivar} and \ref{table:periodic1} shows that in general, a periodic control achieves the target of bringing the population into $\Sigma_-$ at a smaller cost than the constant control (in terms of total number of released mosquitoes, counted with respect to the wild population).

\subsection{Case study: Onetahi motu}

\begin{table}[!ht]\centering
 \begin{adjustbox}{max width=\textwidth}
\begin{tabular}{|l|c|c|c|c|c|}
\hline 
$\nu_E \backslash \beta$ &$10^{-4}$ &$10^{-3}$ &$10^{-2}$ &$10^{-1}$ &$10^{0}$ \\ \hline 
0.001 &  $  39$ & $ 200$ & $ 295$ & $ 376$ & $ 453$ \\ \hline 
0.002 &  $ 142$ & $ 310$ & $ 402$ & $ 480$ & $ 555$ \\ \hline 
0.005 &  $ 877$ & $1094$ & $1178$ & $1252$ & $1323$ \\ \hline 
0.008 &  N/A & N/A & N/A & N/A & N/A \\ \hline 
0.010 &  N/A & N/A & N/A & N/A & N/A \\ \hline 
0.015 &  N/A & N/A & N/A & N/A & N/A \\ \hline 
\end{tabular}
\begin{tabular}{|l|c|c|c|c|c|}
\hline 
$\nu_E \backslash \beta$ &$10^{-4}$ &$10^{-3}$ &$10^{-2}$ &$10^{-1}$ &$10^{0}$ \\ \hline 
0.001 &  $  34$ & $ 181$ & $ 272$ & $ 352$ & $ 430$ \\ \hline 
0.002 &  $ 111$ & $ 262$ & $ 350$ & $ 428$ & $ 503$ \\ \hline 
0.005 &  $ 350$ & $ 471$ & $ 554$ & $ 627$ & $ 697$ \\ \hline 
0.008 &  $1167$ & $1091$ & $1168$ & $1238$ & $1305$ \\ \hline 
0.010 &  N/A & N/A & N/A & N/A & N/A \\ \hline 
0.015 &  N/A & N/A & N/A & N/A & N/A \\ \hline 
\end{tabular}
\begin{tabular}{|l|c|c|c|c|c|}
\hline 
$\nu_E \backslash \beta$ &$10^{-4}$ &$10^{-3}$ &$10^{-2}$ &$10^{-1}$ &$10^{0}$ \\ \hline 
0.001 &  $  30$ & $ 171$ & $ 261$ & $ 341$ & $ 418$ \\ \hline 
0.002 &  $  97$ & $ 241$ & $ 327$ & $ 404$ & $ 480$ \\ \hline 
0.005 &  $ 260$ & $ 381$ & $ 462$ & $ 535$ & $ 605$ \\ \hline 
0.008 &  $ 443$ & $ 541$ & $ 618$ & $ 687$ & $ 754$ \\ \hline 
0.010 &  $ 676$ & $ 728$ & $ 802$ & $ 870$ & $ 935$ \\ \hline 
0.015 &  N/A & N/A & N/A & N/A & N/A \\ \hline 
\end{tabular}
 \end{adjustbox}
 \caption{Entrance time into the basin of $\0$ (in days) for various values of $(\nu_E,\beta)$ with constant weekly ($T = 7 \text{ days}$) releases at $p = 4$ (left), $p = 6$ (center) or $p = 8$ (right).}
 \label{table:TetiaroaEntTim}
\end{table}

We now parametrize explicitly our model to the case of Onetahi motu in Tetiaroa atoll (French Polynesia), where weekly ($T = 7 \text{ days}$) releases have been performed over a year. Male population was estimated at $69 \cdot 74 \simeq 5000$ individuals, and the initial effort ratio $p := \Lambda/M_+^*$ was estimated at~$8$.

\begin{table}[!ht]\centering
 \begin{adjustbox}{max width=\textwidth}
\begin{tabular}{|l|c|c|c|c|c|}
\hline 
$\nu_E \backslash \beta$ &$10^{-4}$ &$10^{-3}$ &$10^{-2}$ &$10^{-1}$ &$10^{0}$ \\ \hline 
0.001 &  $0.943252$ & $0.147678$ & $0.020134$ & $0.002495$ & $0.000283$ \\ \hline 
0.002 &  $0.567382$ & $0.071552$ & $0.009875$ & $0.001247$ & $0.000141$ \\ \hline 
0.005 &  $0.205116$ & $0.031070$ & $0.004439$ & $0.000568$ & $0.000069$ \\ \hline 
0.008 &  $0.133889$ & $0.021388$ & $0.003170$ & $0.000425$ & $0.000052$ \\ \hline 
0.010 &  $0.111803$ & $0.018284$ & $0.002779$ & $0.000380$ & $0.000047$ \\ \hline 
0.015 &  N/A & N/A & N/A & N/A & N/A \\ \hline 
\end{tabular}
 \end{adjustbox}
 \caption{Final total female ratio $\frac{(F+F_{st})(t)}{F^* + F^*_{st}}$ at time $t$ when the trajectory enter into the basin of $\0$ for various values of $(\nu_E,\beta)$ with constant weekly ($T = 7 \text{ days}$) releases at $p = 8$.}
 \label{table:TetiaroaFemRat}
\end{table}

\begin{figure}[!ht]
 \includegraphics[width=.5\textwidth]{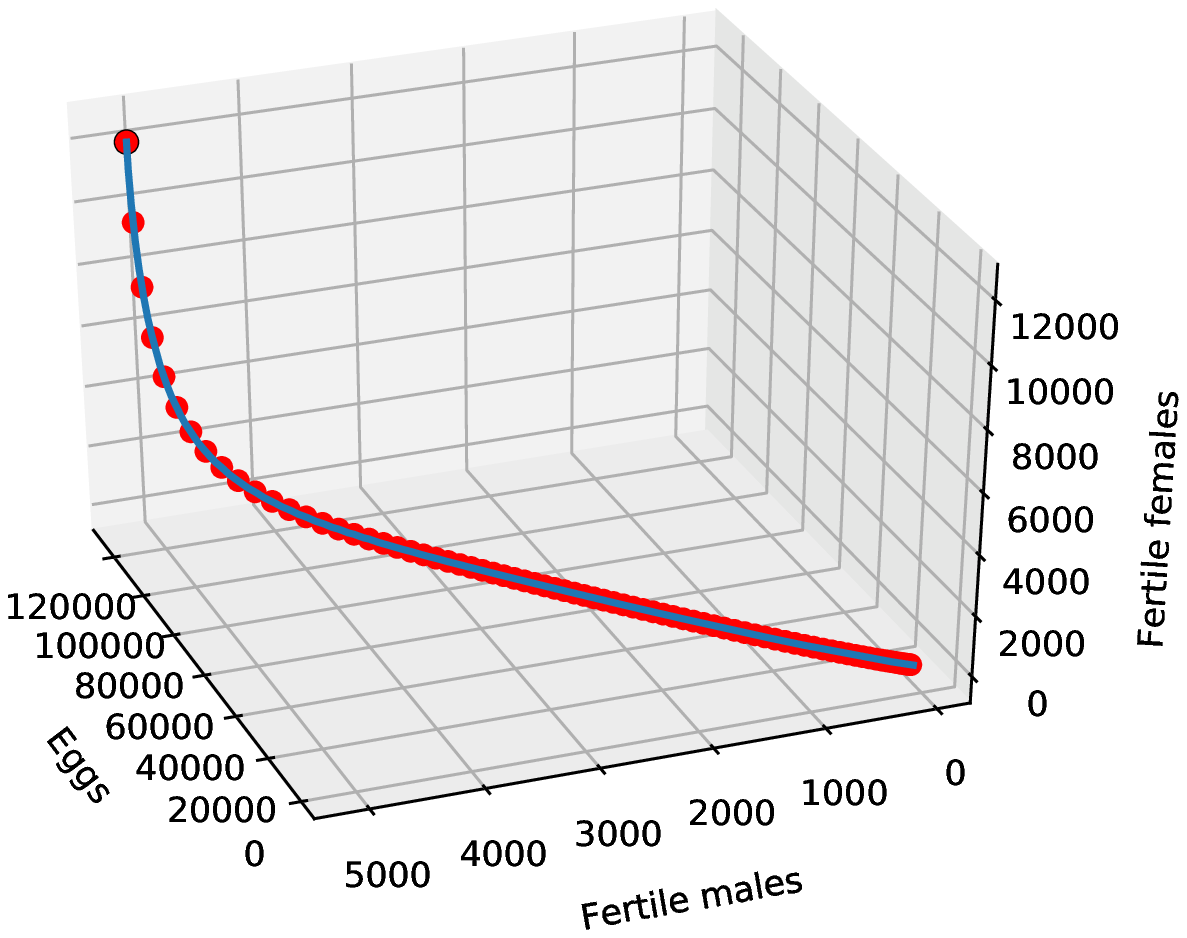}
 \includegraphics[width=.5\textwidth]{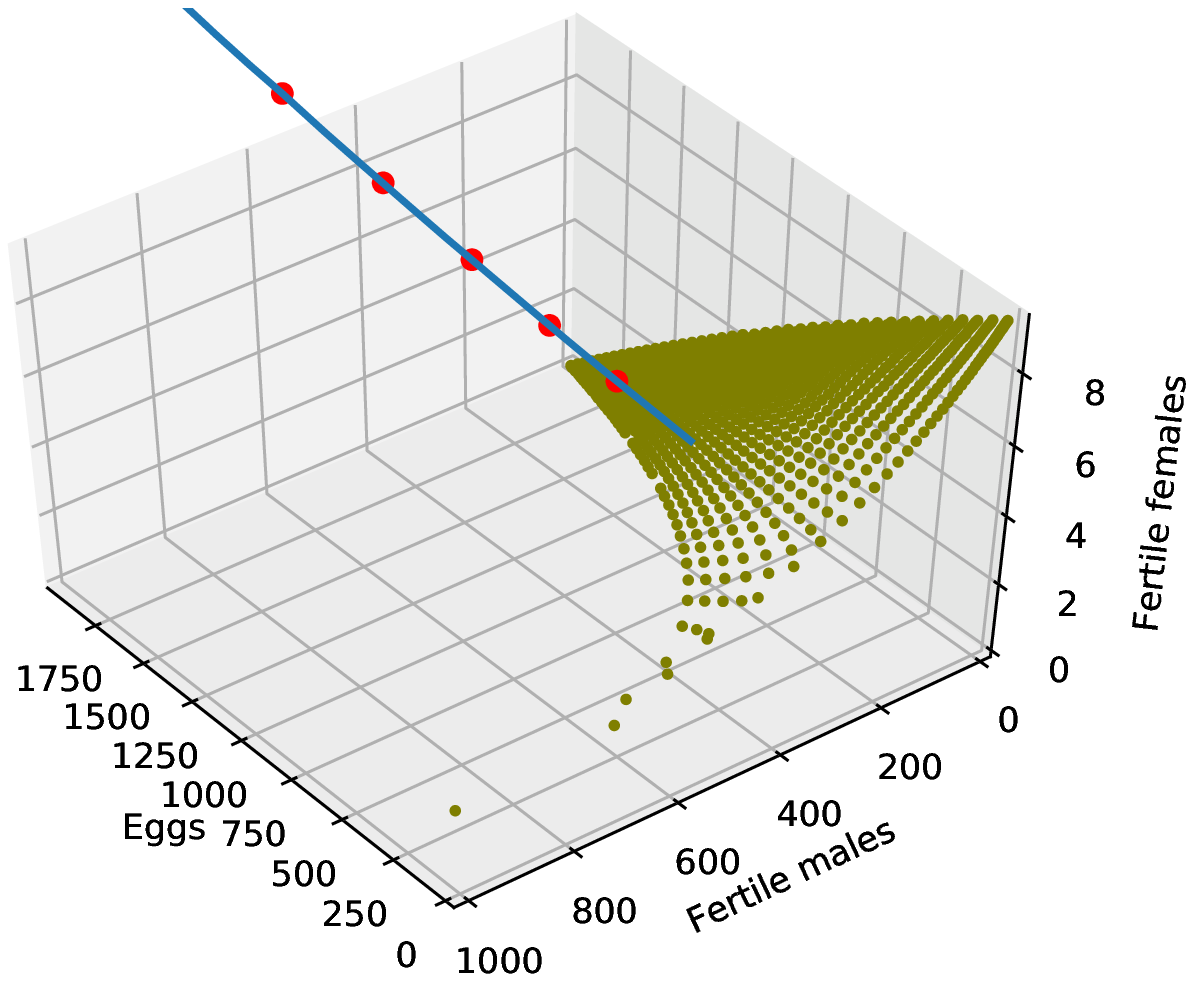}
 \caption{Trajectory $t \mapsto (E(t),M(t),F(t))$ for $\nu_E = 0.008$ and $\beta = 10^{-3}$ (left) and a zoom in the last $30$ days of treatment displaying also the separatrix as dots (right).}
 \label{fig:3dTraj}
\end{figure}

\begin{figure}[!ht]
\begin{subfigure}[t]{.5\textwidth}
 \includegraphics[width=\textwidth]{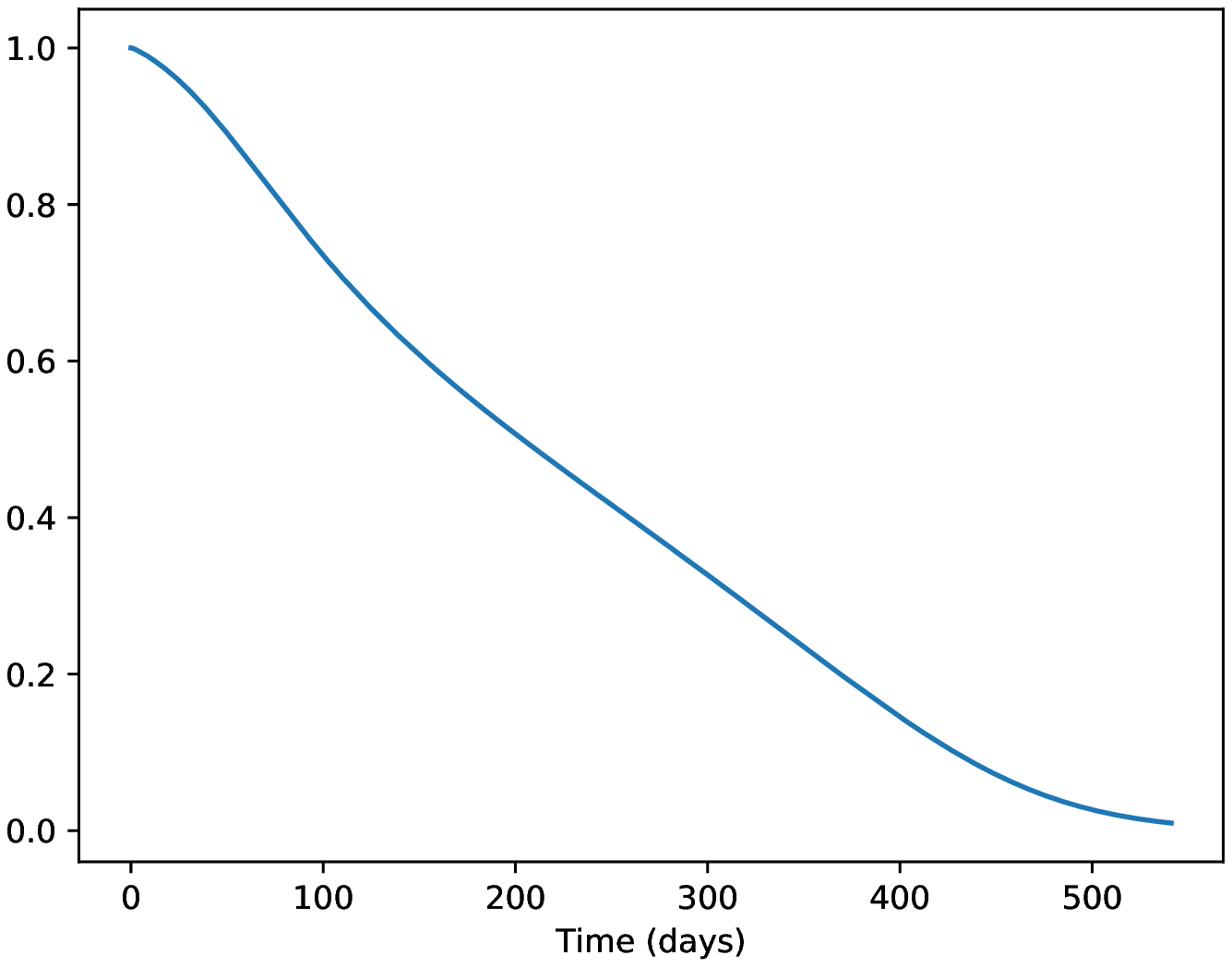}
 \caption{Relative egg count $\frac{E(t)}{E_+^*}$.}
 \label{fig:tetieggs}
\end{subfigure}
\begin{subfigure}[t]{.5\textwidth} 
 \includegraphics[width=\textwidth]{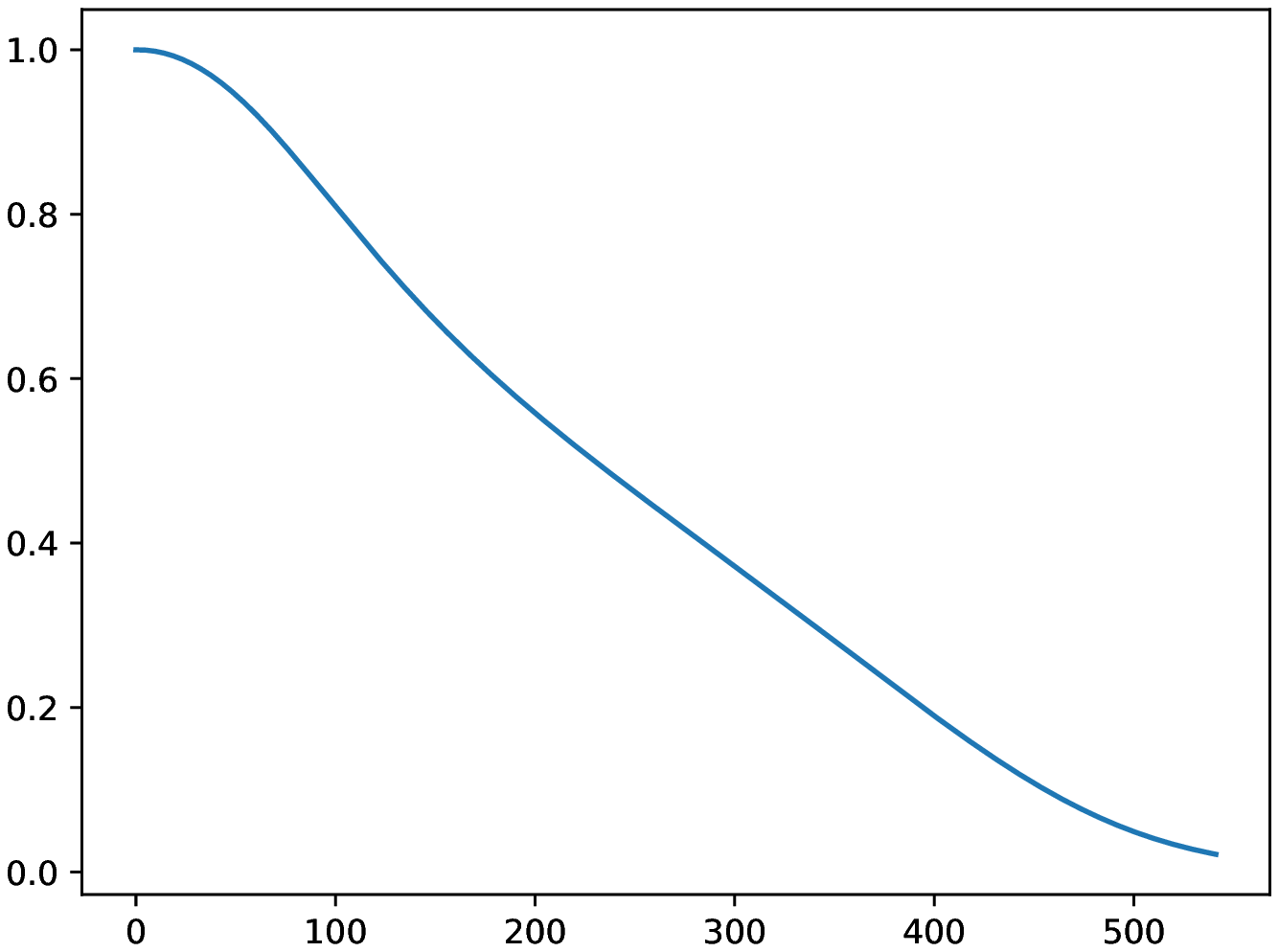}
 \caption{Relative total female count $\frac{F(t)+F_{st}(t)}{F_+^* + F_{st,+}^*}$.}
 \label{fig:tetiFemRat}
\end{subfigure}
\begin{subfigure}[b]{.5\textwidth}
 \includegraphics[width=\textwidth]{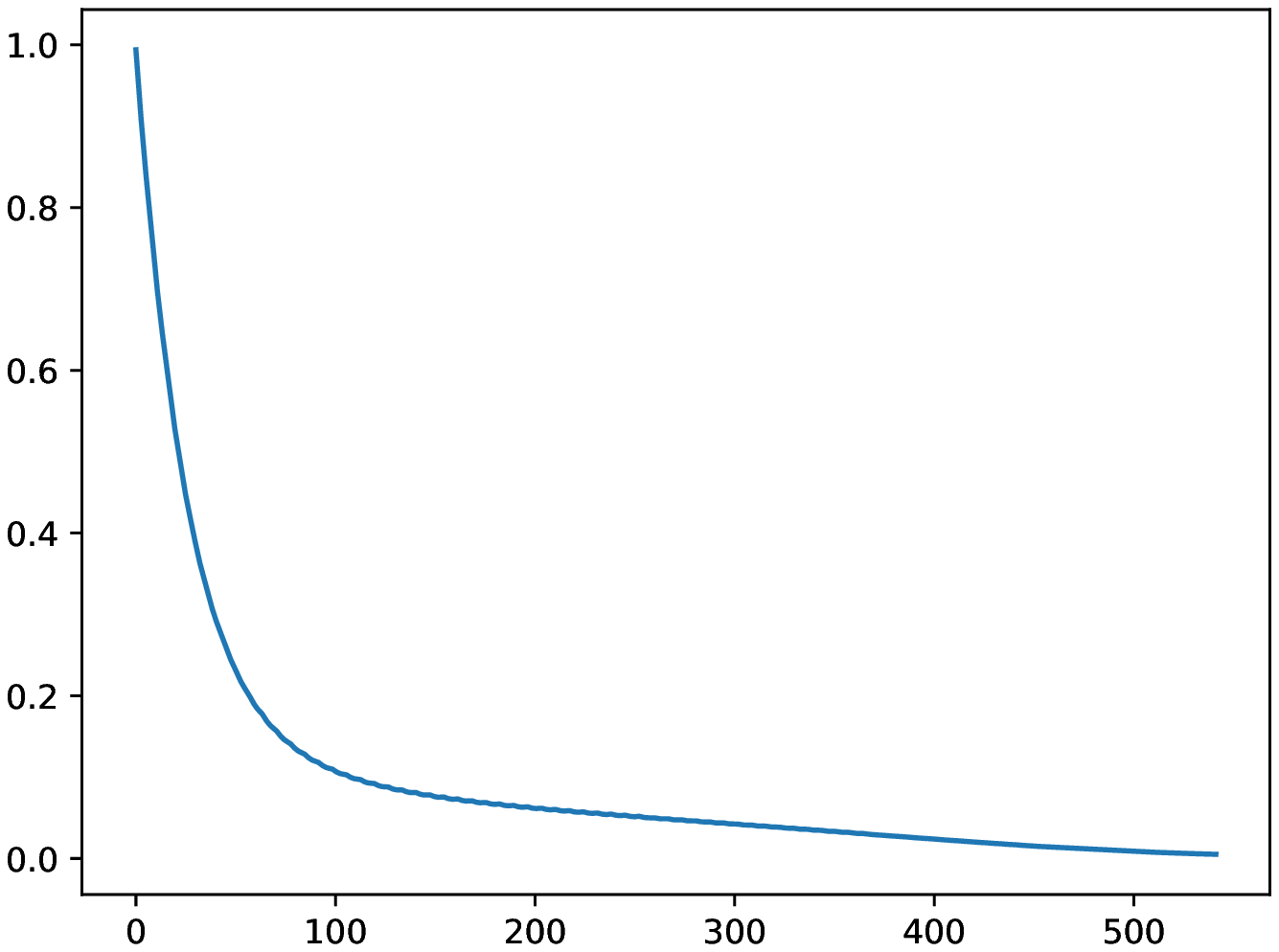}
 \caption{Fertile female ratio $\frac{F(t)}{F(t) + F_{st} (t)}$.}
 \label{fig:tetiFerFemRat}
\end{subfigure}
\begin{subfigure}[b]{.5\textwidth}
 \includegraphics[width=\textwidth]{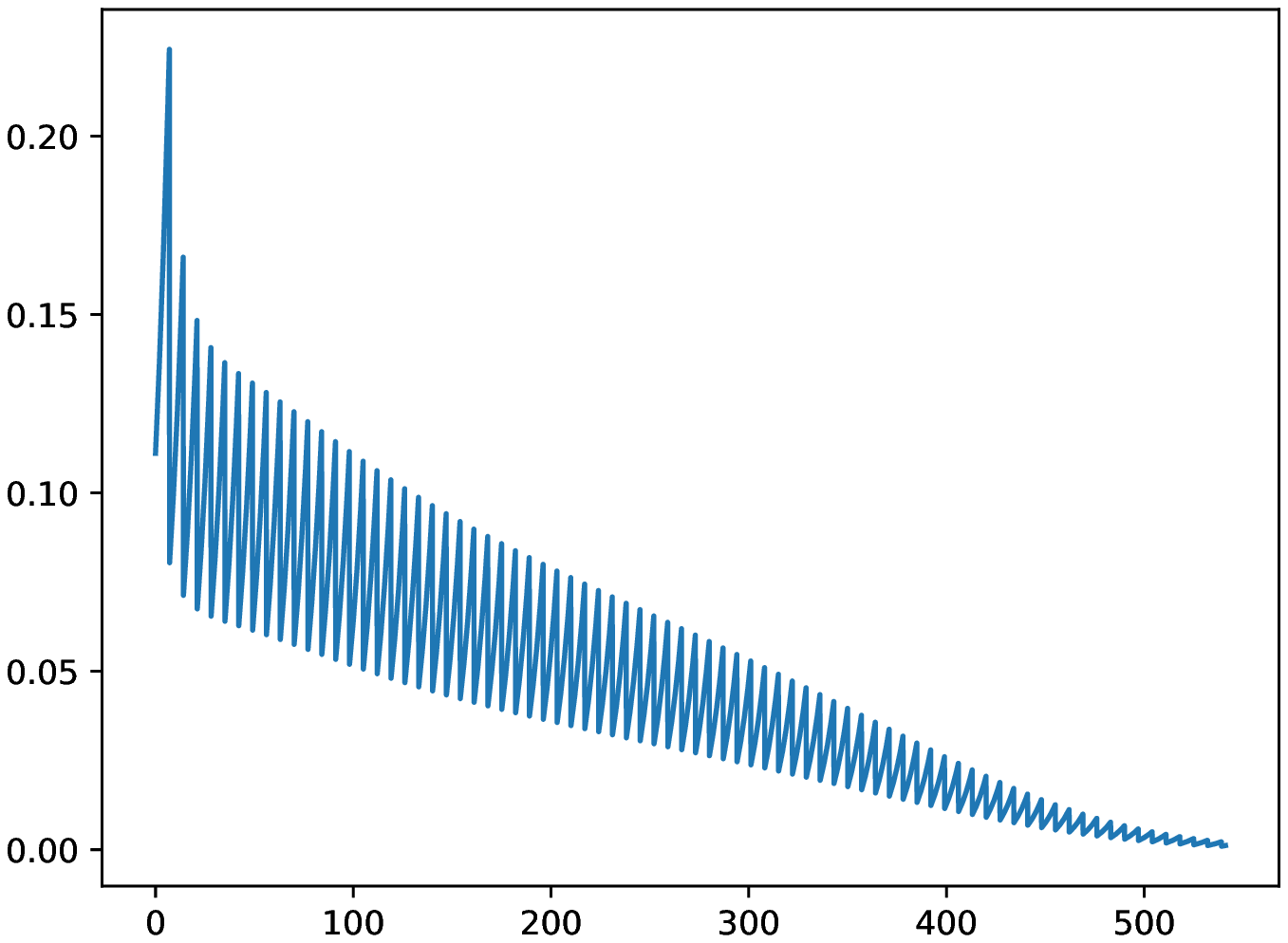}
 \caption{Fertile male ratio $\frac{M(t)}{M(t)+M_i(t)}$.}
 \label{fig:tetiFerMal}
\end{subfigure}
 \caption{Time dynamics of different ratio for $\nu_E = 0.008$ and $\beta = 10^{-3}$.}
 \label{fig:normTimeDynamics}
\end{figure}

For $p \in \{4, 6, 8\}$, entrance times (in days) are shown in Table \ref{table:TetiaroaEntTim} and final total female ratio in Table \ref{table:TetiaroaFemRat}. This last quantity is important for practical purposes to help answering the question: when is it time to stop the releases? The trap counts during the experiment are to be compared with the initial trap counts (before the releases), and roughly, the process can be stopped once the ratio between the counts goes below the values in Table \ref{table:TetiaroaFemRat}. Interestingly, $\beta$ determines the order of magnitude of this final ratio.

Table \ref{table:TetiaroaEntTim} provides us interesting information on the entrance time versus the transition rate $\nu_E$ and the mating parameter $\beta$. If the effort ratio $p$ is not large enough, the SIT treatment can fail, and even if it is large enough (say $p=8$) the time to reach the basin of $\0$ can be very large.

In the $3$-dimensional state space $(E,M,F)$ we draw the full trajectory for the same sample value $(\nu_E = 0.008, \, \beta = 10^{-3}, \, p = 8)$ along with a zoom in the last $30$ days of treatment showing also the separatrix between the basins of $\EE_+$ and $\0$ as dots in Figure \ref{fig:3dTraj}. 
According to Table \ref{table:TetiaroaEntTim}, page \pageref{table:TetiaroaEntTim}, the entrance time is $541$, which justifies that the control should last for more than one year. Our system being monotone, the trajectory is monotone decreasing (see Figure \ref{fig:3dTraj} (left), page \pageref{fig:3dTraj}). However, the rate of the decrease is relatively large at the beginning of the treatment, and then becomes small and, almost, constant.
We also show time dynamics of four relevant normalized quantities, for the same sample value $(\nu_E = 0.008, \, \beta = 10^{-3}, \, p =8)$ in Figure \ref{fig:normTimeDynamics}.

\section{Conclusion}
In this paper we have derived a minimalistic model to control mosquito population by Sterile Insect Technique, using either irradiation or the cytoplasmic incompatibility of {\it Wolbachia} to release sterilizing males. 
We particularly focus on the chance of collapsing the wild population, provided that the selected area allows elimination. Thus contrary to previous SIT and IIT models, the trivial equilibrium, $\0$ is always Locally Asymptotically Stable, at least.  
We consider different type of releases (constant, continuous, or periodic and instantaneous) and show necessary conditions to reach elimination, in each case. We also derived the minimal time under which elimination cannot occur, ({\it i.e.} entrance into the basin of attraction of $\0$ is impossible), whatever the control effort. Obviously, the knowledge on the mosquito parameters are very important, particularly the duration of the egg compartment, $\dfrac{1}{\mu_E+\nu_E}$ and the mating parameter, $\beta$. Surprisingly, mosquito entomologists have not yet really focused their experiments on $\beta$ or the probability of meeting/mating between one male and one female according to the size of the domains. Our model illustrates the importance of this parameter (and others) in the duration of the SIT control. In general, SIT entomologists recommend to release a minimum of ten times more sterile males than (estimated) wild males: this can be necessary if the competitiveness of the sterile male is weak compared to the wild ones (this can be the case with irradiation SIT approach). Our approach may help standardizing and quantifying this estimated ratio.

Finally, we focus on a real case scenario, the Onetahi motu, where a {\itshape Wolbachia} experiment has been conducted by Bossin and collaborators, driving the local mosquito, {\itshape Aedes polynesiensis}, to nearly elimination. Our preliminary results show some good agreement with field observation (mainly trapping).

Our results also show the importance of eggs in the survival of the wild population. If the egg stock is sufficiently large, and depending on weather parameters, the wild population can re-emerge after the control has stopped. That is why, according to our model and numerical results, it is recommended to pursue the release of sterilizing males even after wild mosquito females are no longer collected in monitoring traps.

Last but not least, we hope that our theoretical results will be helpful to improve future SIT experiments and particularly to take into account the long term dynamics of eggs.

\vspace{1cm}
\noindent{\paragraph{Acknowledgments:} This study was part of the “Phase 2 SIT feasibility project against Aedes albopictus in Reunion Island”, jointly funded by the Regional Council of La Reunion and the European Regional Development Fund (ERDF) under the 2014–2020 Operational Program. YD was (partially) supported by the DST/NRF SARChI Chair M3B2 grant 82770. YD and MS also acknowledge partial supports from the STIC AmSud "MOSTICAW Project", Process No. 99999.007551/2015- 00.}
\newpage

%
%
%
\appendix
\section{Proof of Lemma \ref{lem:steadystates}}
\label{app:Lemproof}

\subsection{Study of \texorpdfstring{$f$}{f}}
We first study function $f$ defined in \eqref{fun:f}.
For any $y \geq 0$, if $x \geq \f{1}{\psi}$ then $f(x, y) < - \f{1}{\mathcal{N}} (x + y)$ so in particular $f(x, y) < 0$. Therefore all steady states must satisfy $\beta M^* < \f{1}{\psi}$.
Likewise,
\[
y \geq 0, \, 0 \leq x < \f{1}{\psi} \implies (1 - \psi x) (1 - e^{-(x + y)}) < 1.
\]
Hence for all $x < \f{1}{\psi}$ we find $f(x, y) < (1 - \f{1}{\mathcal{N}}) x - \f{1}{\mathcal{N}} y$.
As a consequence, if $\mathcal{N} \leq 1$ then $f(x, y) < 0$ for all $(x, y) \in \RR_+^2 \backslash \{ 0\}$, and system \eqref{sys:simplifie} has no positive steady state.
From now on we assume that $\mathcal{N} > 1$.

We also compute directly $f(0, y) = - \f{1}{\mathcal{N}} y < 0$ and $ \lim_{x \to +\infty} f(x, y) = - \infty$.

\begin{remark}
 For all $x \in (0, 1/\psi)$, we notice that
\[
f(x, y) < Q_y (x) = - \psi x^2 + ( 1 - \f{1}{\mathcal{N}}) x - \f{y}{\mathcal{N}}.
\]
The discriminant of the second-order polynomial $Q_y$ is
\[
\Delta_y = (1 - \f{1}{\mathcal{N}})^2 - \f{4 y \psi}{\mathcal{N}}.
\]
Let $\widetilde{y} := \f{\mathcal{N}}{4 \psi} (1 - \f{1}{\mathcal{N}})^2$. If $y \geq \widetilde{y}$ then $\Delta_y \leq 0$, hence $f < 0$.
At this stage we know that if $\beta \gamma_i M_i \geq \widetilde{y}$ then there is no positive steady state. 

The quantity $\widetilde{y}$ is used in Remark \ref{rem:ratio} to obtain a first-order approximation of the target release ration.
\end{remark}

We now compute the derivatives of $f$:
\begin{align*}
\p_x f &= (1 - 2 \psi x) (1 - e^{-(x + y)}) - \f{1}{\mathcal{N}} +  x (1 - \psi x) e^{-(x + y)}, \\
\p^2_{xx} f &= - 2 \psi (1 - e^{-(x+y)}) + e^{-(x + y)} \big(  2 - (4 \psi + 1) x + \psi x^2\big) \\
\p^3_{xxx} f &= e^{-(x + y)} \Big( -6 \psi - 3 + (6 \psi + 1) x - \psi x^2 \Big) =: e^{-(x + y)} Q_3 (x)\\
\p_y f &= x (1 - \psi x)  e^{- (x+y)} - \f{1}{\mathcal{N}}, \\
\p^2_{yy} f &= - x (1 - \psi x) e^{-(x + y)} < 0 \text{ for } x \in (0, 1/\psi).
\end{align*}

Obviously, $\p_x f(x, y) < 0$ if $x \geq \f{1}{\psi}$ and $\p_x f(0, y) = 1 - e^{- y} - \f{1}{\mathcal{N}}$, which is positive if and only if $ y > - \log (1 - \f{1}{\mathcal{N}}) = \log(1 + \f{1}{\mathcal{N}-1})$.

In order to know the variations of $\p^3_{xxx} f$ we study the second-order polynomial
\[
Q_3 (x) = - 6 \psi - 3 \beta + x \big( 6 \psi + 1\big) - \psi x^2.
\]
Its discriminant is
\[
\Delta_3 = (6 \psi + 1)^2 - 4 \psi (6 \psi + 3) = 1 + 12 \psi^2,
\]
which is positive. Therefore $\p^3_{xxx} f$ is negative-positive-negative. More precisely, $Q_3$ is positive on 
\[
(w_-, w_+) := \big( \f{ 6 \psi + 1 - \sqrt{1 + 12 \psi^2}}{2 \psi}, \f{ 6 \psi + 1 + \sqrt{1 + 12 \psi^2}}{2 \psi} \big).
\]

To go one step further, we need to know the signs of $\p^2_{xx} f(w_+, y)$ and $\p^2_{xx} f(0, y)$.
We write
\[
\p^2_{xx} f(x, y) > 0 \iff e^{-(x + y)} \Big(2 + 2 \psi - (4 \psi + 1) x +  \psi x^2 \Big) > 2 \psi
\]
Hence $\p^2_{xx} f(0, y) > 0$ if and only if $y < \log( 1 + \f{1}{\psi})$.
Similarly, $\p^2_{xx} f(w_+, y) < 0$ if and only if
\[
y > \log \big(1 + \f{1}{\psi} - (2 + \f{1}{2 \psi})w_+ + \f{1}{2} w_+^2 \big) - w_+.
\]
This is always true:
\begin{lemma}
For all $\psi > 0$,
\[
 \log \big(1 + \f{1}{\psi} - (2 + \f{1}{2 \psi})w_+ + \f{1}{2} w_+^2 \big) - w_+ < 0.
\]
\end{lemma}
\begin{proof}
To prove it, we introduce $\gamma = \f{1}{2 \psi}$ so that we are left with
\[
\log \big(7 + 3 \gamma + \gamma^2 +  (4 + \gamma) \sqrt{3+\gamma^2} \big) - (3 + \gamma + \sqrt{3+\gamma^2}) < 0.
\]

To check this we introduce 
\[
g(x) := \log(7+3x + x^2 + (4+x)\sqrt{3+x^2})-(3+x+\sqrt{3+x^2}),
\]
 and we want to prove that $g$ is negative. We compute that the sign of $g'(x)$ is equal to that of
\[
-(4+x)(3+x^2) - 2 x - \sqrt{3+x^2} ( 8 + 2 x + x^2) < 0.
\]
It remains to check that $g(0) = \log(7 + 4 \sqrt{3}) - (3 + \sqrt{3}) < 0$, which is true since
\[
 e^{3+\sqrt{3}} > e^{4} > 2^{4} > 7 + 8 > 7 + 4 \sqrt{3},
\]
where we used $e > 2$ and $1 < \sqrt{3} < 2$.
\end{proof}
Thus we obtain that $x \mapsto \p^2_{xx} f(x, y)$ is either positive-negative (if $y < \log(1 + \f{1}{\psi})$) or always negative (otherwise).

The conclusion of all these computations is that in both cases ($f$ is either convex-concave or simply concave), for any $y$, $f(0, y) < 0$, $f(+\infty, y) = - \infty$ so that all in all there are either $0$, $1$ or $2$ solutions to $f (x, y)= 0$, depending merely on the sign of the maximum of $x \mapsto f(x, y)$.

\subsection{Study of functions \texorpdfstring{$h_{\pm}$}{h}}
\label{fonctionh}
We move on to the next step of the proof, studying the functions $h_{\pm}$ defined in \eqref{fun:hpm}. Recall that solving $f(x, y) = 0$ (for $x, y > 0$) is equivalent to picking $\theta = e^{-(x + y)} \in (0, 1)$ and $y = h_{\pm} (\theta)$.

First, to check that $h_+$ and $h_-$ are well-defined we need to check that $1 + \xi \f{\log(\theta)}{1-\theta} > 0$ for some $\theta \in (0, 1)$. It is easily checked that this is the case on $(\theta_0 (\xi), 1)$, and $\theta_0 (\xi)$ is well-defined as soon as $\xi < 1$.

Hence if $\xi \geq 1$ then there is no nonzero steady state.
Assume therefore that $\xi < 1$. Then there exists a unique $\theta_0 (\xi) \in (0, 1)$ such that $1 - \theta - \xi \log(\theta)$ has the same sign as $\theta-\theta_0$ on $(0, 1)$, that is, $1 - \theta_0 = \f{4 \psi}{\mathcal{N} } \log(\theta_0)$.

We can check that $h_-$ is decreasing, $h_- < h_+$ on $(\theta_0, 1]$,
\[
h_{\pm} (\theta_0) = -\f{1}{2\psi} - \log(\theta_0),
\]
and
\[
h_- (1) < h_+ (1) = \f{1}{2 \psi} \big( - 1 + \sqrt{1 - \xi} \big) < 0.
\]
Indeed (recall that $\mathcal{N} \xi = 4 \psi$),
\[
h'_-(\theta) = - \f{1}{\theta} - \f{1}{\mathcal{N}} \f{\f{1}{\theta(1-\theta)}+\f{\log(\theta)}{(1-\theta)^2}}{\sqrt{1 + \xi \f{\log(\theta)}{1-\theta}}} < 0,
\]
since
\[
-\f{\log(\theta)}{1-\theta} < \f{1}{\theta}.
\]

Let $y^{\text{crit}} := \max_{\theta \in [\theta_0 (\xi), 1]} h_+ (\theta)$. If $y = y^{\text{crit}}$ then there is exactly one solution to $f(x, y) = 0$. For any $y \in [0, y^ {\text{crit}})$, there are at least two solutions. By the previous computations we know that there are at most two solutions. So in this case there are exactly two solutions. To describe them one should consider $I_1 := [ 0, h_- (\theta_0 (\xi)) ]$, if $h_- (\theta_0 (\xi)) > 0$ ($I_1 = \emptyset$ otherwise), and $I_2 = (\max(\theta_0 (\xi), 0), y^{\text{crit}})$. If $y \in I_1$ then there is a solution of the form $h_- (\theta_-)$ and one of the form $h_+ (\theta_-)$. If $y \in I_2$ then both solutions are of the form $h_+ (\theta)$, for two values of $\theta$ whose range contains the argument of $y^{\text{crit}}$.
And for $y > y^{\text{crit}}$ there is no solution.

At this stage we proved that if $\xi \geq 1$ then there is no positive steady state; if $\xi < 1$ then if $y^{\text{crit}} > 0$ then there are two positive steady states for $\beta \gamma_i M_i \in [0, y^{\text{crit}})$, $1$ for $\beta \gamma_i M_i = y^{\text{crit}}$ and $0$ for $\beta \gamma_i M_i > y^{\text{crit}}$. If $y^{\text{crit}} = 0$ then there is a unique positive steady state and if $y^{\text{crit}} < 0$ then there is no positive steady state for any $M_i \geq 0$. 

%

\subsection{Stability}

Finally, in order to compute the linearized stability of the steady states, we decompose $J = M_0 + N_0$, where $M_0$ is non-negative and $N_0$ is diagonal non-positive. Then $J$ (being Metzler, since $E < K$ at steady states) is stable if and only if $\rho(-N_0^{-1} M_0) < 1$. We compute
\[
N_0 = 
\begin{pmatrix}
- \f{bF}{K} - (\nu_E + \mu_E) & 0 & 0 \\
0 & - \mu_M & 0 \\
0 & 0 & - \mu_F
\end{pmatrix}
\]
and
\[
M_0 = 
\begin{pmatrix}
0 & 0 & b (1 - \f{E}{K}) \\
(1 - r) \nu_E & 0 & 0 \\
\f{r \nu_E M}{M+\gamma_i M_i} ( 1 - e^{-\beta (M + \gamma_i M_i)}) & \f{r \nu_E E}{M+\gamma_i M_i} \big(\beta M e^{-\beta (M + \gamma_i M_i)} + \f{\gamma_i M_i}{M + \gamma_i M_i} (1 - e^{-\beta (M + \gamma_i M_i)}) \big) & 0
\end{pmatrix}
\]
so that for some $X_1, X_2 \in \RR$ (which we compute below at steady states) we have
\[
- N_0^{-1} M_0 =
\begin{pmatrix}
0 & 0 & \f{b (1 - \f{E}{K})}{b \f{F}{K} + \nu_E + \mu_E} \\
\f{(1-r)\nu_E}{\mu_M} & 0 & 0 \\
\f{X_1}{\mu_F} & \f{X_2}{\mu_F} & 0
\end{pmatrix}.
\]
At the steady state $(0, 0, 0)$, we have directly unconditional stability as 
\[
J =
\begin{pmatrix}
- (\nu_E + \mu_E) & 0 & b \\
(1 - r) \nu_E & -\mu_M & 0 \\
0 & 0 & - \mu_F
\end{pmatrix},
\]
whose eigenvalues are $-(\nu_E + \mu_E)$, $-\mu_M$ and $-\mu_F$.

At a non-zero steady state we recall that
\begin{align*}
b F &= \f{(\nu_E + \mu_E) E}{1-\f{E}{K}}, \\
E &= \lambda K M, \\
r \nu_E (1 - e^{-\beta (M + \gamma_i M_i)})\f{M}{M+\gamma_i M_i} &= \mu_F \f{F}{E} = \f{\mu_F (\nu_E + \mu_E)}{b} \f{1}{1-\lambda M}, \\
e^{-\beta (M + \gamma_i M_i)} &= 1 - \f{1}{\mathcal{N} ( 1 - \lambda M) } \f{M + \gamma_i M_i}{M},
\end{align*}
so that
\begin{align*}
X_1 &= \f{r \nu_E}{\mathcal{N}(1-\lambda M)}, \\
X_2 &=  \f{r \nu_E \lambda K M}{M + \gamma_i M_i} \Big(  \beta M \big(  1 - \f{M+ \gamma_i M_i}{\mathcal{N}M ( 1 - \lambda M)} \big) +  \f{\gamma_i M_i}{M} \f{1}{\mathcal{N} ( 1 - \lambda M)} \Big).
\end{align*}
The characteristic polynomial of $-N_0^{-1} M_0$ is
\[
P(z) = - z^3 + \f{b ( 1 - \lambda M)^2}{\nu_E+\mu_E}  \Big( \f{(1-r)\nu_E X_2}{\mu_M \mu_F} + z \f{X_1}{\mu_F} \Big),
\]
which is equal to
\[
P(z) = - z^3 + \mathcal{N} ( 1 - \lambda M )^2 \Big( \f{M}{M+\gamma_i M_i} \big( \beta M (1 - \f{M+\gamma_i M_i}{\mathcal{N}M (1 - \lambda M)}) + \f{\gamma_i M_i}{M \mathcal{N} (1 - \lambda M)}\big) + \f{z}{\mathcal{N} ( 1 - \lambda M)} \Big),
\]
and we rewrite it as
\[
P(z) = -z^3 + (1 - \lambda M) \Big(\beta \mathcal{N} \f{M^2 (1 - \lambda M)}{M+\gamma_i M_i} - \beta M + \f{\gamma_i M_i}{M+\gamma_i M_i} + z \Big)
\]
We find $P(0) > 0$ (since $X_2 > 0$) and 
\[
P'(z) = - 3 z^2 + (1 - \lambda M),
\]
so that $J$ is stable if and only if $P(1) < 0$. ($P$ is increasing and then decreasing on $(0, + \infty)$). This condition reads
\beq
(1 - \lambda M) \Big(1 + \f{\gamma_i M_i}{M+\gamma_i M_i} + \beta M \big( -1 + \mathcal{N} \f{M}{M+\gamma_i M_i} (1 - \lambda M) \big) \Big) < 1.
\label{cond:stab}
\eeq

Let us treat first the case when $M_i = 0$. The stability condition rewrites
\[
(1 - \lambda M) \big(1 + \beta M (-1 + \mathcal{N} (1 - \lambda M)) \big) < 1,
\]
that is, for a nonzero steady state,
\[
-\lambda + \beta (-1 + \mathcal{N} (1 - \lambda M)) - \lambda \beta M (-1 + \mathcal{N}(1 - \lambda M)) < 0.
\]

If $M_i^{\text{crit}} > 0$, we know that there are exactly two steady states between $0$ and $1/\lambda$ for $M_i = 0$, which we denote by $0 < M_- < M_+ < 1/\lambda$. Let $\phi(x) = 1 - \f{1}{\mathcal{N}} - \lambda x + e^{-\beta x} (\lambda x - 1)$. We have $\phi(M_{\pm}) = 0$ and $\mp \phi'(M_{\pm})) > 0$.

In particular, $\phi' (M_+) > 0$ so
\[
M_+ > \f{1}{\lambda} + \f{1}{\beta} (1 - e^{\beta M_+}) = \f{1}{\lambda} - \f{1}{\beta} \f{1}{(1-\lambda M_+) \mathcal{N} - 1}.
\]
Multiplying this inequality by $\lambda \beta \big( (1 -\lambda M_+) \mathcal{N} - 1 \big)$ yields exactly the stability of $M_+$, since $(1 -\lambda M_+) \mathcal{N} > 1$. Indeed,
\[
\mathcal{N} (1 - \lambda M_{\pm}) = \f{e^{\beta M_{\pm}}}{e^{\beta M_{\pm}} -1} > 1. 
\]
By a similar computation one can show that the smaller steady state $M_-$ is unstable.

We move now to the general case $M_i \geq 0$, assume $M_i < M_i^{\text{crit}}$ and write that $\p_x f < 0$ (which was proved to hold at the bigger steady state) is equivalent to
\[
(1 - 2 \lambda M) (1 - e^{-\beta (M + \gamma_i M_i)}) + \beta M (1 - \lambda M)e^{-\beta (M + \gamma_i M_i)} < \f{1}{\mathcal{N}}.
\]
Using as before the fact that $M$ is a steady state allows us to rewrite this last inequality as
\[
(1-2 \lambda M) \f{1}{\mathcal{N}} \f{M+\gamma_i M_i}{M(1-\lambda M)} + \beta M (1 - \lambda M) \big(1 - \f{M+ \gamma_i M_i}{\mathcal{N} (1-\lambda M) M} \big) < \f{1}{\mathcal{N}}.
\]
Multiplying this inequality by $\mathcal{N} (1 - \lambda M)\f{M}{M+\gamma_i M_i}$ yields
\[
(1-2 \lambda M)  + \beta (1 -\lambda M) \big(\mathcal{N}M^2 \f{1 - \lambda M}{M+\gamma_i M_i}- M  \big) < (1 - \lambda M)\f{M}{M+\gamma_i M_i},
\]
that is
\[
(1 - \lambda M) \Big(2 - \f{M}{M+\gamma_i M_i} + \beta M \big( - 1 + \mathcal{N} M \f{1-\lambda M}{M + \gamma_i M_i} \big) \Big) < 1,
\]
whence the stability of the bigger steady state, since we recover \eqref{cond:stab}.
Likewise, at the smaller steady state we have $\p_x f > 0$, and the reverse inequality holds. This concludes the proof.

\section{Basin entrance time approximation}
\label{sec:entrancetime}

\subsection{Bounds on the wild equilibria}

For $M_i = 0$, under the assumptions of Lemma \ref{lem:steadystates} such that there are two positive steady states $\EE_- \ll \EE_+$ for \eqref{sys:simplifie}, we get explicit bounds on these states. In particular, we assume $\calN > 4 \psi$.
We recall that the positive equilibria can be expressed as an increasing function of their second coordinate $M \in (0, 1/\lambda)$:
\[
 \EE(M) := \begin{pmatrix} 
            K \lambda M
            \\
            M
            \\
            \f{\nu_E + \mu_E}{b} \f{\lambda M}{(1- \lambda M)}
           \end{pmatrix},
\]
and $\EE(M)$ is an equilibrium if and only if $f (\beta M) = 0$, where
\begin{equation}
 f(x) = (1 - \psi x)(1 - e^{-x}) - \f{1}{\calN}.
 \label{def:fMi0}
\end{equation}

\begin{lemma}
 The function $f$ (defined in \eqref{def:fMi0}) is concave on $[0, 1/\psi]$. It reaches its maximum value on this interval at $Z(\psi) \in (0, \f{1}{2\psi})$, where we define
 \begin{equation}
  e^{-Z(\psi)} = \f{\psi}{1+\psi - \psi Z(\psi)}, \quad F(\psi) := \f{1+\psi-\psi Z(\psi)}{(1 - \psi Z(\psi))^2}.
  \label{def:Fpsi}
 \end{equation}
 Then $f$ on $[0, 1/\psi]$ has no zero if $\calN < F(\psi)$, exactly $1$ zero if $\calN = F(\psi)$ and exactly $2$ zeros if $\calN > F(\psi)$.
 
 In addition, $Z$ and $F$ have the following asymptotics:
 \[
  Z(\psi) \sim_{\psi \to +\infty} \f{1}{2 \psi}, \quad Z(\psi ) \sim_{\psi \to 0} \log\big( \f{1}{\psi} \big), \quad F(\psi) \sim_{\psi \to +\infty} 4 \psi, \quad F \xrightarrow[\psi \to 0]{} 1.
 \]
 \label{lem:ZF}
 \end{lemma}
\begin{proof}
 We compute
 \[
  f'(x) = e^{-x} \big( 1 + \psi - \psi x \big) - \psi, \quad f''(x) = e^{-x} \big( \psi x - 1 - 2 \psi \big),
 \]
 hence $f'' < 0$ on $[0, 1/\psi]$. Since $f(0) = f(1/\psi) = - 1/\calN < 0$, $f$ reaches a unique maximum at the (necessarily unique) point $Z(\psi) \in (0, 1/\psi)$ such that $f'(Z(\psi)) = 0$.  
 The claim that $Z(\psi) < 1/(2 \psi)$ follows from the inequality $e^x > 1+x$, which implies that
 \[
  \f{1}{\psi} f'(\f{1}{2 \psi}) = e^{-1/2\psi} \big( 1 + \f{1}{2 \psi} \big) - 1 < 0.
 \]
 Moreover, the sign of $f(Z(\psi))$ is exactly that of $\calN - F(\psi)$.
 The equivalents and limit follow from straightforward computations.
\end{proof}
\begin{remark}
We notice that $Z$ is related to a well-known special function: let us introduce the (principal branch of the) special Lambert $W$ function, that is:
\[
 W(y) = z, \, z \geq - 1 \iff z e^{z} = y.
\]
Since if $y > 1$ then $z > 0$, we obtain
\[
 Z(\psi) = \log \big( W (e^{1+1/\psi}) \big).
\]
\end{remark}

Assume $\calN > F(\psi)$ (defined in \eqref{def:Fpsi}), and denote by $x_- < x_+$ the two positive zeros of $f$.

\begin{lemma}
We have $x_- > 1/\calN$. 
 \[
  \f{1}{\calN} < x_- < \f{1}{\psi} \big( 1 - \f{\kappa^*}{\calN} \big) < Z(\psi) < \f{1}{\psi} \big( 1 - \f{\kappa_*}{\calN} \big) < x_+,
 \]
 where
 \[
 \kappa_* = 1 + \f{\psi}{1-\psi Z (\psi)}, \quad \kappa^* = \calN - \f{\psi Z(\psi) (1 + \psi - \psi Z (\psi))}{(1 - \psi Z(\psi))^2}.
\]
If in addition $\calN > 2$ then $x_+ < \f{1}{\psi} \big(1 - \f{1}{\calN} \big)$.
\label{lem:boundsx}
\end{lemma}
\begin{proof}
 The first inequality is obtained by using the inequalities $1 - e^{-x} \leq x$ and $1 - \sqrt{1-x} > x / 2$ for $x \in (0, 1)$.
 The first one implies that $f(x) \leq x (1 - \psi x) - 1/\calN$, which is a second order polynomial equal to $f$ at $0$ and at $1/\psi$, with roots located at $\big(1\pm \sqrt{1-4\psi/\calN}\big)/(2\psi)$ (recall that we have $\calN > 4 \psi$).
 Hence $x_- > \big(1 - \sqrt{1-4\psi/\calN}\big)/(2\psi) > 1/\calN$ by the second inequality.
 
 The upper bound on $x_+$ comes from the fact that if $\calN > 2$ then by Lemma \ref{lem:ZF}
 \[
  \big( 1 - \f{1}{\calN} \big) \f{1}{\psi} > \f{1}{2 \psi} > Z(\psi).
 \]

 Finally to get the two other bounds, we introduce
 \[
  H(\kappa) := f \big( \f{1}{\psi} ( 1 - \f{\kappa}{\calN}) \big) = \kappa \big( 1 - e^{-\f{1}{\psi} (1 - \f{\kappa}{\calN})} \big) - 1.
 \]
 By Lemma \ref{lem:ZF}, it is concave on $[0, \calN]$, equal to $-1$ at $0$ and $\calN$ and reaches its maximum at $\widehat{\kappa} := \calN ( 1 - \psi Z(\psi))$.
 To get $\kappa_*$ and $\kappa^*$, we simply use the fact that the graph of $H$ is above the segments from $(0, -1)$ to $(\widehat{\kappa}, H(\widehat{\kappa}))$ on the first hand, and from $(\widehat{\kappa}, H(\widehat{\kappa}))$ to $(\calN, 0)$ on the other hand, so that we define
 \[
  -1 + \f{H(\widehat{\kappa})+1}{\widehat{\kappa}} \kappa_* = 0 = -1 - (\kappa^* - \calN) \f{H(\widehat{\kappa}) + 1}{\calN - \widehat{\kappa}},
 \]
 and the expressions of $\kappa_* < \widehat{\kappa} < \kappa^*$ follow from a straightforward computation.

\end{proof}

Back to the steady states of \eqref{sys:simplifie}, we deduce from Lemma \ref{lem:boundsx} the following bounds, assuming $\calN > 2$:
\begin{align}
 \undest{\EE}_- := \begin{pmatrix}
             \frac{\lambda K }{\mathcal{N}\beta} \\
             \frac{1}{\mathcal{N} \beta} \\
             \frac{\nu_E+\mu_E}{b} \frac{\lambda K}{\mathcal{N} \beta} 
            \end{pmatrix} \leq 
           & \EE_- \leq \big( 1 - \frac{\kappa^*}{\calN} \big) \begin{pmatrix}
             K  \\
             \frac{1}{\lambda} \\
             \frac{\nu_E+\mu_E}{b} \frac{K \calN }{\kappa^*}
            \end{pmatrix} =: \ovest{\EE}_- 
            \label{estimations:EE-}
            \\
            \undest{\EE}_+ := \big(1 - \frac{\kappa_*}{\calN} \big)\begin{pmatrix}
             K  \\
             \frac{1}{\lambda} \\
             \frac{\nu_E+\mu_E}{b} \frac{K \calN}{\kappa_*}
            \end{pmatrix} \leq & \EE_+ \leq \big( 1 - \frac{1}{\mathcal{N}} \big) \begin{pmatrix} K \\ \frac{1}{\lambda} \\ \frac{K \mathcal{N} (\nu_E + \mu_E)}{b} \end{pmatrix} =: \ovest{\EE}_+. 
            \label{estimations:EE+}
\end{align}

\subsection{Results}

\paragraph{A lower bound.} First, we give a lower bound on the entrance times. We consider the fact that for a solution to \eqref{sys:simplifie} with initial data given by $\EE_+$, thanks to the overestimation in \eqref{estimations:EE+},
\[
 F(t) \geq \undest{F}_+ e^{-\mu_F t} =: \undest{F}_{\flat} (t).
\]
This implies 
\[
 E(t) \geq e^{-(\nu_E + \mu_E) t - \frac{b \undest{F}_+}{K} ( 1 - e^{-\mu_F t}) } \undest{E}_+ + b \undest{F}_+ \int_0^t e^{-\mu_F t'} e^{-(\nu_E+\mu_E)(t-t')} e^{-\frac{b \undest{F}_+}{K} ( e^{-\mu_F t'} - e^{-\mu_F t})} dt' =: \undest{E}_{\flat}(t),
\]
and
\[
 M(t) \geq e^{-\mu_M t} \undest{M}_+ + (1 - r) \nu_E \int_0^t e^{-\mu_M ((t - t')} \undest{E}_{\flat}(t') dt' =: \undest{M}_{\flat}(t).
\]
Using the underestimation of $\EE_-$ from \eqref{estimations:EE-}, we define $t^{Z}_{\flat} := \min \{ t \geq 0, \quad \undest{Z}_{\flat}(t) \leq \ovest{Z}_- \}$ for $Z \in \{ E, M, F \}$.
\begin{lemma}
 We have the following lower bound: $\tau(M_i) \geq \min \big( t^E_{\flat}, t^M_{\flat}, t^F_{\flat} \big)$.
\end{lemma}

Explicitly we find, with $Z = Z(\psi)$ and $Z_0 = 1 + \psi - \psi Z$:
\[
 t^F_{\flat} = \f{1}{\mu_F} \log \big( \f{\kappa^* (\calN - \kappa_*)}{\kappa_* (\calN - \kappa^*}\big) = \f{1}{\mu_F} \log \big( 1 + \f{\calN^2 (1 - \psi Z)^3}{\psi Z Z_0^2} - \f{\calN (1 - \psi Z)}{\psi Z Z_0} \Big).
\]
However it must be expected that $\min(t^E_{\flat}, t^M_{\flat}) > t^F_{\flat}$, and we can give explicit approximations of $t^E_{\flat}$ and $t^M_{\flat}$.

\paragraph{A first upper bound.} 
We compare the solution of \eqref{sys:simplifie} with the solution of the linear system
 \beq
 \bepa
 \df{d E_e}{dt} = b F_e - (\nu_E + \mu_E) E_e,
\\[10pt]
\df{d M_e}{dt} = (1-r) \nu_E E_e - \mu_M M_e,
\\[10pt]
\df{dF_e}{dt} = r \nu_E \epsilon(M_i) E_e - \mu_F F_e,
 \label{sys:estimate}
 \eepa
 \eeq
 where $\epsilon(M_i) = \max_{t \geq 0} \frac{M(t)}{M(t) + M_i} < 1$, typically $\epsilon(M_i) = \frac{M^*}{M^* + M_i}$.
 The following property follows from the fact that \eqref{sys:simplifie} is cooperative:
 \begin{lemma}
 Solutions of \eqref{sys:simplifie} and \eqref{sys:estimate} with initial data such that $(E^0, M^0, F^0) \leq (E_e^0, M_e^0, F_e^0)$ satisfy:
 \[
 \forall t \geq 0, \, (E(t), M(t), F(t)) \leq (E_e (t), M_e (t), F_e(t)).
 \]
 \label{lem:sysestimate}
 \end{lemma}
 We use the under-estimation of $\EE_-$ given by \eqref{estimations:EE-}, to define, for $X = (X^i)_i = (E, M, F)$ and $i \in \{1, 2, 3\}$,
 \[
t^{X^i}_{\min} := \inf \{ t \geq 0, X^i_e(t) \leq [\undest{\EE}_- ]_i \}.  
\]
\begin{lemma}
 For any solution $X_e$ to \eqref{sys:estimate} satisfying the assumption of Lemma \ref{lem:sysestimate}, we have the upper bound on the entrance time: $\tau(M_i) \leq \max \big( t^E_{\min}, t^M_{\min}, t^F_{\min} \big)$.
 \label{lem:firstub}
\end{lemma}

Analytic computations are made in Section \ref{sec:analytic}.

\paragraph{An second upper bound in two steps.}
Let $\rho^* := M_i / \ovest{M}_+$ be the under-estimated effort ratio. When using the above one-step approach, we conclude with a finite upper bound for $\tau(M_i)$ if and only if $\ovest{M}_+ / (M_i + \ovest{M}_+) < 1/\mathcal{N}$, that is
\begin{equation}
 \rho^* > \mathcal{N} - 1.
 \label{cond:onesteprhostar}
\end{equation}

Expanding upon the same idea as for the lower bound, we let $\epsilon = \ovest{M}_+ / (\ovest{M}_+ + M_i)$ so
\[
 F(t) \leq \ovest{F}_+ e^{-\mu_F t} + \ovest{E}_+ r \nu_E \epsilon (1 - e^{-\mu_F t}) =: \ovest{F}_{\sharp}.
\]

Then, we construct the explicit solution $(E,M) = (\ovest{E}_{\sharp}, \ovest{M}_{\sharp})$ to
\[
 \begin{array}{l}
  \dot{E} = b \ovest{F}_{\sharp} - \big( \nu_E + \mu_E + \frac{\ovest{F}_{\sharp}}{K} \big) E, \quad E(0) = \ovest{E}_+,
  \\[10pt]
  \dot{M} = (1 - r) \nu_E E - \mu_M M, \quad M(0) = \ovest{M}_+.
 \end{array}
\]

In details:
\[
 \begin{array}{l}
  \ovest{F}_{\sharp}(t) = \ovest{E}_+ r \nu_E \epsilon + e^{-\mu_F t} \big( \ovest{F}_+ - r\nu_E \epsilon \ovest{E}_+ \big),
  \\[10pt]
  \ovest{E}_{\sharp}(t) = e^{-(\nu_E+\mu_E + \frac{\ovest{E}_+ r \nu_E \epsilon}{K})t - \frac{\ovest{F}_+ - r \nu_E \epsilon \ovest{E}_+}{K \mu_F} (1-e^{-\mu_F t})} \Big( \ovest{E}_+ + \int_0^t \big( b \ovest{E}_+ r \nu_E \epsilon 
  \\[10pt]
  + b e^{-\mu_F t'} ( \ovest{F}_+ - r\nu_E \epsilon \ovest{E}_+ ) \big) e^{(\nu_E+\mu_E + \frac{\ovest{E}_+ r \nu_E \epsilon}{K})t' - \frac{\ovest{F}_+ + r \nu_E \epsilon \ovest{E}_+}{K \mu_F} (1-e^{-\mu_F t'})} dt' \Big),
  \\[10pt]
  \ovest{M}_{\sharp}(t) = e^{-\mu_F t} \ovest{M}_+ + (1-r) \nu_E \int_0^t e^{\mu_F t'} \ovest{E}_{\sharp}(t') dt'.
 \end{array}
\]

We use this super-solution on $[0, t_0]$ (for some $t_0 > 0$ to be determined), and then glue the solution on $[t_0, +\infty)$ of
\[
\left\{
 \begin{array}{l}
  \dot{E} = b F - (\nu_E + \mu_E) E, \quad E(t_0) = \ovest{E}_{\sharp} (t_0),
  \\[10pt]
  \dot{M} = (1- r) \nu_E E - \mu_M M, \quad M(t_0) = \ovest{M}_{\sharp} (t_0),
  \\[10pt]
  \dot{F} = r \nu_E \epsilon_0 E - \mu_F F, \quad F(t_0) = \ovest{F}_{\sharp} (t_0),
 \end{array}
\right.
\]
with $\epsilon_0 = \ovest{M}_{\sharp} (t_0) / (\ovest{M}_{\sharp}(t_0) + M_i) < \epsilon$.

For $Z \in \{E, M, F\}$ we let 
\[
 t^Z_{\sharp} (t_0) := \min \{ t \geq t_0, \ovest{Z}_{\sharp} \leq \undest{Z}_- \}.
\]
Then as before:
\begin{lemma}
 For all $t_0 > 0$, $\tau(M_i) \leq t_{\sharp}(t_0) := \max ( t^E_{\sharp}(t_0), t^M_{\sharp}(t_0), t^F_{\sharp}(t_0) )$.
 \label{lem:secondub}
\end{lemma}
By using Lemma \ref{lem:secondub}, we can theoretically obtain a finite upper bound for $\tau(M_i)$ (upon choosing a suitable $t_0$) as soon as $\epsilon_0 < 1/\mathcal{N}$ for $t_0$ large enough, that is if and only if
\begin{equation}
 \rho^* \big( (\rho^* + 1) \frac{\mu_M}{(1-r)\nu_E} + \mathcal{N}-1) > \mathcal{N} - 1.
 \label{cond:twosteprhostar}
\end{equation}
Condition \eqref{cond:twosteprhostar} is weaker than \eqref{cond:onesteprhostar} (and in general, much weaker). It holds if and only if
\[
 \rho^* > \frac{-(\mathcal{N}-1+\phi) + \sqrt{ (\mathcal{N}-1+\phi)^2+ 4 \phi (\mathcal{N}-1)}}{2 \phi}, \quad \phi := \lambda K = \f{\mu_M}{(1-r)\nu_E},
\]
which is true for instance if $\rho^* > \sqrt{(\mathcal{N}-1)/\phi}$.
However, we do not develop any further these analytic computations in the present paper.

\subsection{Analytic computations}
\label{sec:analytic}
 Applying Lemma \ref{lem:firstub}, in order to express analytically the solution $X_e := (E_e, M_e, F_e)$ of \eqref{sys:estimate}, we only need to diagonalize the matrix
 \[
 R_e := \begin{pmatrix}
 - (\nu_E + \mu_E) & b \\
 r \nu_E \epsilon & - \mu_F
 \end{pmatrix}.
 \]
 $R_e$ has negative trace, and positive determinant if and only if $\frac{1}{\mathcal{N}} > \epsilon$.
 Hence if $ \mathcal{N}\epsilon (M_i) < 1$ then $\0$ is globally asymptotically stable for \eqref{sys:estimate}.

 In this case its eigenvalues are real, negative and equal to $\kappa_{\pm}$ associated respectively with eigenvectors $\begin{pmatrix} 1 \\ x_{\pm} \end{pmatrix}$, where
 \begin{align*}
 \kappa_{\pm} &:= \frac{-(\nu_E + \mu_E + \mu_F) \pm \sqrt{(\nu_E + \mu_E - \mu_F)^2 + 4 b r \nu_E \epsilon}}{2}, 
 \\
 x_{\pm} &:= \frac{\nu_E + \mu_E - \mu_F \pm \sqrt{(\nu_E + \mu_E - \mu_F)^2 + 4 b r \nu_E \epsilon}}{2 b}.
 \end{align*}
 Then we deduce that for some real numbers $(r_{\pm}^0, s_{\pm}^0) \in \RR^4$,
 \begin{align*}
 E_e(t) &= r_+^0 e^{\kappa_+ t} + r_-^0 e^{\kappa_- t}, \\
 F_e(t) &= s_+^0 e^{\kappa_+ t} + s_-^0 e^{\kappa_- t}, \\
 M_e(t) &= e^{-\mu_M t} M_e^0 + (1-r) \nu_E \int_0^t e^{-\mu_M(t-t')} \big(r_+^0 e^{\kappa_+ t'} + r_-^0 e^{\kappa_- t'} \big) dt'.
 \end{align*}
 In details, we find
 \begin{align*}
  &r_+^0 = \f{x_-}{x_- - x_+} E_e^0 - \f{1}{x_- - x_+} F_e^0, &r_-^0 = \f{-x_+}{x_- - x_+}E_e^0 + \f{1}{x_- - x_+} F_e^0
  \\
  &s_+^0 = \f{x_+ x_-}{x_- - x_+} E_e^0 - \f{x_+}{x_- - x_+} F_e^0, &s_-^0 = \f{-x_+ x_-}{x_- - x_+} E_e^0 + \f{x_-}{x_- - x_+} F_e^0.
 \end{align*}

 Assuming $\kappa_+ \not= -\mu_M$ and $\kappa_- \not= -\mu_M$ (which must hold generically since these are biological parameters), we get
 \[
 M_e(t) = e^{-\mu_M t} M_e^0 + (1-r) \nu_E \big(r_+^0 \frac{e^{\kappa_+ t} - e^{-\mu_M t}}{\mu_M + \kappa_+} + r_-^0 \frac{e^{\kappa_- t} - e^{-\mu_M t}}{\mu_M + \kappa_-}\big).
 \]

Assuming $\calN > 2$, we use the overestimation \eqref{estimations:EE+} of $\EE_+$ as an initial data $(E_e^0, M_e^0, F_e^0)$, and with the notations
\[
 g(\epsilon) = \sqrt{1+\frac{4 b r \nu_E \epsilon}{(\nu_E + \mu_E - \mu_F)^2}},  \quad \sigma = \sgn(\nu_E + \mu_E - \mu_F),
\]
we deduce
\begin{align*}
 r^0_{\pm} &= \f{K}{2}\big( 1 - \f{1}{\calN} \big) \Big( 1 \pm \f{(2 \calN - 1) (\nu_E + \mu_E) + \mu_F}{g(\epsilon) \lvert \nu_E + \mu_E - \mu_F \rvert} \Big),
 \\
 s^0_{\pm} &= \f{K \lvert \nu_E + \mu_E - \mu_F \rvert}{4 b g(\eps)}\big( 1 - \f{1}{\calN} \big) \big( \sigma \pm g(\eps) \big) \big( g(\eps) \pm \f{(2 \calN - 1) (\nu_E + \mu_E) + \mu_F}{\lvert \nu_E + \mu_E - \mu_F \rvert} \big).
\end{align*}

If $r_-^0 < 0$ the we can use the simple upper bound $E_e (t) \leq r_+^0 e^{\kappa_+ t}$. This condition reads
\[
 g(\epsilon) \lvert \nu_E + \mu_E - \mu_F \rvert < (2 \mathcal{N} - 1) (\nu_E + \mu_E) + \mu_F.
\]

In this case, we know that $E_e (t) \leq \big[ \undest{\EE}_- \big]_1$ if $r_+^0 e^{\kappa_+ t} \leq \frac{\lambda K}{\mathcal{N} \beta}$, that is if
\begin{equation}
 t \geq t^E_{\min} := \frac{2}{\nu_E + \mu_E + \mu_F - g(\epsilon) \lvert \nu_E + \mu_E - \mu_F \rvert} \log \Big( \frac{(\mathcal{N} - 1)}{2\psi} \big( 1 + \frac{(2 \mathcal{N} - 1) (\nu_E + \mu_E) + \mu_F}{g(\epsilon) \lvert \nu_E + \mu_E - \mu_F \rvert} \big) \Big)
 \label{eq:tEmin}
\end{equation}

Then, under the same condition we have $s_{\pm}^0 > 0$. By using the fact that $s_+^0 + s_-^0 = \ovest{F}_+$, we deduce that $F_e (t) \leq \big[ \undest{\EE}_- \big]_3$ if $\ovest{F}_+ e^{\kappa_+ t} \leq \undest{F}_-$, that is if
\begin{equation}
 t \geq t^F_{\min} := \frac{2}{\nu_E + \mu_E + \mu_F - g(\epsilon) \lvert \nu_E + \mu_E - \mu_F \rvert} \log \Big(\f{\calN (\calN - 1)}{\psi} \Big).
 \label{eq:tFmin}
\end{equation}

In addition, we have $t^E_{\min} > t^F_{\min}$ if and only if
\[
 (2 \calN - 1) (\nu_E + \mu_E) + \mu_F > (\calN - 1) g(\eps) \lvert \nu_E +\mu_E - \mu_F \rvert.
\]

\begin{remark}
 For small $\epsilon$, the previous estimations roughly show that
 \[
  t_{\min} \geq \frac{1}{\min(\nu_E+\mu_E,\mu_F)} \log \big( \f{\calN^2}{\psi} \big).
 \]
\end{remark}

Finally, we need to compute the condition $M_e(t) \leq \frac{1}{\mathcal{N} \beta}$.
Let $\sigma_E := \mu_M / (\nu_E + \mu_E)$ and $\sigma_F := \mu_M / \mu_F$. We rewrite $M_e(t)$ as
 \[
  M_e (t) = \f{1}{\lambda} \big(1 - \f{1}{\calN} \big) \big( \alpha e^{-\mu_M t} + \alpha_+ e^{\kappa_+ t} + \alpha_- e^{\kappa_- t} \big),
 \]
 with
 \[
  \alpha = \f{(\calN - 1) \sigma_F + 1 - \eps \calN}{(\sigma_F - 1) (\sigma_E - 1) - \eps \calN}, \quad \alpha_{\pm} = \f{\mu_M}{\mu_M + \kappa_{\pm}}\widetilde{r}_{\pm}^0 ,
 \]
 where
 \[
  \widetilde{r}_{\pm}^0 := \f{1}{2} \big( 1 \pm \f{2 \calN - 1 + \sigma_E / \sigma_F}{g(\epsilon) \sigma (1 - \sigma_E/\sigma_F)} \big), \quad
  g(\eps) = \sqrt{1 + \f{4 \calN \sigma_E \sigma_F \eps}{(\sigma_F -  \sigma_E)^2}}
 \]
  and
 \[
  \f{\mu_M}{\mu_M + \kappa_{\pm}} = \f{2 \sigma_E \sigma_F}{2 \sigma_E \sigma_F - (\sigma_E + \sigma_F) \pm \sigma (\sigma_F - \sigma_E) g(\eps)}.
 \]
 The condition we need to compute is therefore
 \[
  \alpha e^{-\mu_M t} + \alpha_+ e^{\kappa_+ t} + \alpha_- e^{\kappa_- t} \leq \f{\psi}{\calN - 1}.
 \]
We assume that the male half-life is shorter than that of the females and of the eggs, so that $\sigma_F, \sigma_E > 1$. Under the stronger assumptions that $r_- < 0 < r_+$ and
\[
\eps \calN < 1, \quad (\sigma_F - 1)(\sigma_E - 1) > \eps \calN,
\]
we obtain that $\alpha > 0$.
We simply treat two subcases: first if $\mu_M + \kappa_+ < 0$ (small $\mu_M$) then we obtain $\alpha_+ < 0 < \alpha_-$ and thus
\[
 t^M_{\min} := \f{1}{\mu_M} \log \Big( (\calN - 1) \f{\alpha + \alpha_-}{\psi} \Big).
\]
Second, if $\mu_M + \kappa_- > 0$ (large $\mu_M$) then we obtain $\alpha_- < 0 < \alpha_+$ and thus
\[
 t^M_{\min} := \f{1}{-\kappa_+} \log \Big( (\calN - 1) \f{\alpha + \alpha_+}{\psi} \Big).
\]

In the last case (when $\mu_M$ is large), we can check that $t^M_{\min} > t^E_{\min}$ is equivalent to
\[
 \alpha + \alpha_+ > \widetilde{r}_+^0,
\]
which holds since $\alpha > 0$ and $\alpha_+ > \widetilde{r}_+^0$.

In this case we obtain
\begin{multline*}
 \max \big( t^E_{\min}, t^F_{\min}, t^M_{\min} \big) = t^M_{\min} 
 \\
 = \frac{2 \sigma_E}{\mu_F \big(\sigma_F + \sigma_E - g(\epsilon) \sigma (\sigma_F - \sigma_E) \big)} \log \Big(\f{\calN-1}{\psi} \big(\f{(\calN - 1) \sigma_F + 1 - \eps \calN}{(\sigma_F - 1) (\sigma_E - 1) - \eps \calN} 
 \\
 + \f{\sigma_E \sigma_F \big( g(\eps) \sigma (\sigma_F - \sigma_E) + (2 \calN - 1)\sigma_F + \sigma_E  \big)}{\big( 2 \sigma_E \sigma_F - (\sigma_E + \sigma_F) + \sigma (\sigma_F - \sigma_E) g(\eps)\big) g(\eps) \sigma (\sigma_F - \sigma_E)} \big) \Big).
\end{multline*}

\bibliographystyle{siam}
\bibliography{biblio}
\end{document}